\documentclass[a4paper]{amsart}

\usepackage{amscd}
\usepackage{amsmath}
\usepackage{amssymb}
\usepackage{amsthm}
\usepackage{bbm}

\usepackage[T1]{fontenc}
\usepackage[utf8]{inputenc}

\newif\ifpdf
\ifx\pdfoutput\undefined
   \pdffalse        
\else
   \pdfoutput=1     
   \pdftrue
\fi

\ifpdf
   \usepackage[pdftex]{graphicx}
\else
   \usepackage{graphicx}
\fi

\numberwithin{equation}{section} \swapnumbers

\newtheorem{satz}{Satz}[section]

\newtheorem{theorem}[satz]{Theorem}
\newtheorem{proposition}[satz]{Proposition}
\newtheorem{corollary}[satz]{Corollary}
\newtheorem{lemma}[satz]{Lemma}

\newtheorem{definition}[satz]{Definition}

\newtheorem{remark}[satz]{Remark}

\newtheorem{example}[satz]{Example}

\begin{document}

\hyphenation{ti-me--ho-mo-ge-ne-ous}

\title[Stochastic partial differential equations]{Foundations of the theory of semilinear stochastic partial differential equations}
\author{Stefan Tappe}
\address{Leibniz Universit\"{a}t Hannover, Institut f\"{u}r Mathematische Stochastik, Welfengarten 1, 30167 Hannover, Germany}
\email{tappe@stochastik.uni-hannover.de}
\thanks{The author is grateful to Daniel Gaigall, Georg Grafendorfer, Florian Modler and Thomas Salfeld for valuable comments and discussions.}
\begin{abstract}
The goal of this review article is to provide a survey about the foundations of semilinear stochastic partial differential equations. In particular, we provide a detailed study of the concepts of strong, weak and mild solutions, establish their connections, and review a standard existence- and uniqueness result. The proof of the existence result is based on a slightly extended version of the Banach fixed point theorem.
\end{abstract}
\keywords{Stochastic partial differential equation, solution concepts, existence- and uniqueness result, invariant manifold}
\subjclass[2010]{60H15, 60G17}
\maketitle\thispagestyle{empty}

\tableofcontents

\section{Introduction}

Semilinear stochastic partial differential equations (SPDEs) have a broad spectrum of applications including natural sciences and economics.
The goal of this review article is to provide a survey about the foundations of SPDEs, which have been presented in the monographs \cite{Da_Prato, Prevot-Roeckner, Atma-book}. It may be beneficial for students who are already aware about stochastic calculus in finite dimensions and who wish to have survey material accompanying the aforementioned references. In particular, we review the relevant results from functional analysis about unbounded operators in Hilbert spaces and strongly continuous semigroups.

A large part of this article is devoted to a detailed study of the concepts of strong, weak and mild solutions to SPDEs, to establish their connections, and to review and prove a standard existence- and uniqueness result. The proof of the existence result is based on a slightly extended version of the Banach fixed point theorem.

In the last part of this article we study invariant manifolds for weak solutions to SPDEs. This topic does not belong to the general theory of SPDEs, but it uses and demonstrates many of the results and techniques of the previous sections. It arises from the natural desire to express the solutions of SPDEs, which generally live in an infinite dimensional state space, by means of a finite dimensional state process, and thus, to ensure larger analytical tractability.

This article should also serve as an introductory article to the general theory of SPDEs and enable the reader to learn about further topics and generalizations in this field. Possible further directions are the study of martingale solutions (see, e.g., \cite{Da_Prato, Atma-book}), SPDEs with jumps (see, e.g., \cite{P-Z-book} for SPDEs driven by L\'{e}vy processes and, e.g., \cite{Ruediger-mild, SPDE, Marinelli-Prevot-Roeckner, Tappe-Refine} for SPDEs driven by Poisson random measures), and support theorems as well as further invariance results for SPDEs, see, e.g., \cite{Nakayama-Support, Nakayama}.

The remainder of this article is organized as follows: In Sections~\ref{sec-operators} and \ref{sec-semigroups} we review the required results from functional analysis. In particular, we collect the relevant material about unbounded operators and strongly continuous semigroups. In Section~\ref{sec-processes} we review stochastic processes in infinite dimension. In particular, we recall the definition of a trace class Wiener process and outline the construction of the It\^{o} integral. In Section~\ref{sec-solution-concepts} we present the solution concepts for SPDEs and study their various connections. In Section~\ref{sec-convolution} we review results about the regularity of stochastic convolution integrals, which is essential for the study of mild solutions to SPDEs. In Section~\ref{sec-existence} we review a standard existence- and uniqueness result. Finally, in Section~\ref{sec-manifolds} we deal with invariant manifolds for weak solutions to SPDEs.

\section{Unbounded operators in Hilbert spaces}\label{sec-operators}

In this section, we review the relevant properties about unbounded operators. We shall start with operators in Banach spaces, and focus on operators in Hilbert spaces later on. The reader can find the proofs of the upcoming results in any textbook about functional analysis, such as \cite{Rudin} or \cite{Werner}.

Let $X$ and $Y$ be Banach spaces. For a linear operator $A : X \supset \mathcal{D}(A) \rightarrow Y$, defined on some subspace $\mathcal{D}(A)$ of $X$, we call $\mathcal{D}(A)$ the \emph{domain} of $A$.

\begin{definition}\label{def-closed}
A linear operator $A : X \supset \mathcal{D}(A) \rightarrow Y$ is called \emph{closed}, if for every sequence $(x_n)_{n \in \mathbb{N}} \subset \mathcal{D}(A)$, such that the limits $x = \lim_{n \rightarrow \infty} x_n \in X$ and $y = \lim_{n \rightarrow \infty} A x_n \in Y$ exist, we have $x \in \mathcal{D}(A)$ and $A x = y$.
\end{definition}

\begin{definition}\label{def-dense}
A linear operator $A : X \supset \mathcal{D}(A) \rightarrow Y$ is called \emph{densely defined}, if its domain $\mathcal{D}(A)$ is dense in $X$, that is $\overline{\mathcal{D}(A)} = X$.
\end{definition}

\begin{definition}\label{def-resolvent}
Let $A : X \supset \mathcal{D}(A) \rightarrow X$ be a linear operator.
\begin{enumerate}
\item The \emph{resolvent set} of $A$ is defined as
\begin{align*}
\rho(A) := \{ \lambda \in \mathbb{C} : \lambda - A : \mathcal{D}(A) \rightarrow X \text{ is bijective and } (\lambda - A)^{-1} \in L(X) \}.
\end{align*}

\item The \emph{spectrum of $A$} is defined as $\sigma(A) := \mathbb{C} \setminus \rho(A)$.

\item For $\lambda \in \rho(A)$ we define the \emph{resolvent $R(\lambda, A) \in L(X)$} as
\begin{align*}
R(\lambda, A) := (\lambda - A)^{-1}.
\end{align*}

\end{enumerate}
\end{definition}

Now, we shall introduce the adjoint operator of a densely defined operator in a Hilbert space. Recall that a for a bounded linear operator $T \in L(H_1,H_2)$, mapping between two Hilbert spaces $H_1$ and $H_2$, the adjoint operator is the unique bounded linear operator $T^* \in L(H_2,H_1)$ such that
\begin{align*}
\langle Tx,y \rangle_{H_2} = \langle x,T^* y \rangle_{H_1} \quad \text{for all $x \in H_1$ and $y \in H_2$.}
\end{align*}
In order to extend this definition to unbounded operators, we recall the following extension result for linear operators.

\begin{proposition}\label{prop-lin-op}
Let $X$ be a normed space, let $Y$ be a Banach space, let $D \subset X$ be a dense subspace and let $\Phi : D \rightarrow Y$ be a continuous linear operator. Then there exists a unique continuous extension $\hat{\Phi} : X \rightarrow Y$, that is, a continuous linear operator with $\hat{\Phi}|_{D} = \Phi$. Moreover, we have $\| \hat{\Phi} \| = \| \Phi \|$.
\end{proposition}

Now, let $H$ be a Hilbert space. We recall the representation theorem of Fr\'{e}chet-Riesz. In the sequel, the space $H'$ denotes the dual space of $H$.

\begin{theorem}\label{thm-frechet-riesz}
For every $x' \in H'$ there exists a unique element $x \in H$ with $\langle x', \bullet \rangle = \langle x, \bullet \rangle$. In addition, we have $\| x \| = \| x' \|$.
\end{theorem}

Let $A : H \supset \mathcal{D}(A) \rightarrow H$ be a densely defined operator. We define the subspace
\begin{align}\label{def-domain-adjoint}
\mathcal{D}(A^*) := \{ y \in H : x \mapsto \langle Ax, y \rangle \text{ is continuous on } \mathcal{D}(A) \}.
\end{align}
Let $y \in \mathcal{D}(A^*)$ be arbitrary. By virtue of the extension result for linear operators (Proposition~\ref{prop-lin-op}), the operator 
\begin{align*}
\mathcal{D}(A) \rightarrow \mathbb{R}, \quad x \mapsto \langle Ax, y \rangle
\end{align*}
has a unique extension to a linear functional $z' \in H'$. By the representation theorem of Fr\'{e}chet-Riesz (Theorem~\ref{thm-frechet-riesz}) there exists a unique element $z \in H$ with $\langle z', \bullet \rangle = \langle z, \bullet \rangle$. This implies
\begin{align*}
\langle Ax, y \rangle = \langle x, z \rangle \quad \text{for all $x \in \mathcal{D}(A)$.}
\end{align*}
Setting $A^* y := z$, this defines a linear operator $A^* : H \supset \mathcal{D}(A^*) \rightarrow H$, and we have
\begin{align*}
\langle Ax, y \rangle = \langle x, A^* y \rangle \quad \text{for all $x \in \mathcal{D}(A)$ and $y \in \mathcal{D}(A^*)$.}
\end{align*}

\begin{definition}\label{def-adjoint}
The operator $A^* : H \supset \mathcal{D}(A^*) \rightarrow H$ is called the \emph{adjoint operator} of $A$.
\end{definition}

\begin{proposition}\label{prop-A-adj-dicht}
Let $A : H \supset \mathcal{D}(A) \rightarrow H$ be densely defined and closed. Then $A^*$ is densely defined and we have $A = A^{**}$.
\end{proposition}

\begin{lemma}\label{lemma-domain}
Let $H$ be a separable Hilbert space and let $A : H \supset \mathcal{D}(A) \rightarrow H$ be a closed operator. Then the domain $(\mathcal{D}(A), \| \cdot \|_{\mathcal{D}(A)})$ endowed with the graph norm
\begin{align*}
\| x \|_{\mathcal{D}(A)} = \big( \| x \|^2 + \| Ax \|^2 \big)^{1/2}
\end{align*}
is a separable Hilbert space, too.
\end{lemma}

\section{Strongly continuous semigroups}\label{sec-semigroups}

In this section, we present the required results about strongly continuous semigroups. Concerning the proofs of the upcoming results, the reader is referred to any textbook about functional analysis, such as \cite{Rudin} or \cite{Werner}. Throughout this section, let $X$ be a Banach space.

\begin{definition}
Let $(S_t)_{t \geq 0}$ be a family of continuous linear operators $S_t : X \rightarrow X$, $t \geq 0$.
\begin{enumerate}
\item The family $(S_t)_{t \geq 0}$ is a called a \emph{strongly continuous semigroup} (or \emph{$C_0$-semigroup}), if the following conditions are satisfied:
\begin{itemize}
\item $S_0 = {\rm Id}$,

\item $S_{s+t} = S_s S_t$ for all $s, t \geq 0$,

\item $\lim_{t \rightarrow 0} S_t x = x$ for all $x \in X$.
\end{itemize}

\item The family $(S_t)_{t \geq 0}$ is called a \emph{norm continuous semigroup}, if the following conditions are satisfied:
\begin{itemize}
\item $S_0 = {\rm Id}$,

\item $S_{s+t} = S_s S_t$ for all $s, t \geq 0$,

\item $\lim_{t \rightarrow 0} \| S_t - {\rm Id } \| = 0$.
\end{itemize}
\end{enumerate}
\end{definition}

Note that every norm continuous semigroup is also a $C_0$-semigroup. The following growth estimate (\ref{Halbgruppe-Wachstum}) will often be used when dealing with SPDEs.

\begin{lemma}\label{lemma-wachstum}
Let $(S_t)_{t \geq 0}$ be a $C_0$-semigroup. Then there are constants $M \geq 1$ and $\omega \in \mathbb{R}$ such that
\begin{align}\label{Halbgruppe-Wachstum}
\| S_t \| \leq M e^{\omega t} \quad \text{for all $t \geq 0$.}
\end{align}
\end{lemma}

\begin{definition}
Let $(S_t)_{t \geq 0}$ be a $C_0$-semigroup.
\begin{enumerate}
\item The semigroup $(S_t)_{t \geq 0}$ is called a \emph{semigroup of contractions} (or \emph{contractive}), if
\begin{align}
\| S_t \| \leq 1 \quad \text{for all $t \geq 0$,}
\end{align}
that is, the growth estimate (\ref{Halbgruppe-Wachstum}) is satisfied with $M=1$ and $\omega = 0$.

\item The semigroup $(S_t)_{t \geq 0}$ is called a \emph{semigroup of pseudo-contractions} (or \emph{pseudo-contractive}), if there exists a constant $\omega \in \mathbb{R}$ such that
\begin{align}\label{growth-pseudo}
\| S_t \| \leq e^{\omega t} \quad \text{for all $t \geq 0$,}
\end{align}
that is, the growth estimate (\ref{Halbgruppe-Wachstum}) is satisfied with $M=1$.
\end{enumerate}
\end{definition}

If $(S_t)_{t \geq 0}$ is a semigroup of pseudo-contractions with growth estimate (\ref{growth-pseudo}), then $(T_t)_{t \geq 0}$ given by
\begin{align*}
T_t := e^{-\omega t} S_t, \quad t \geq 0
\end{align*}
is a semigroup of contractions. Hence, every pseudo-contractive semigroup can be transformed into a semigroup of contractions, which explains the term \emph{pseudo-contractive}.

\begin{lemma}\label{lemma-C0-stetig}
Let $(S_t)_{t \geq 0}$ be a $C_0$-semigroup. Then the following statements are true: 
\begin{enumerate}
\item The mapping
\begin{align*}
\mathbb{R}_+ \times X \rightarrow X, \quad (t, x) \mapsto S_t x
\end{align*}
is continuous.

\item For all $x \in X$ and $T \geq 0$ the mapping
\begin{align*}
[0, T] \rightarrow X, \quad t \mapsto S_t x
\end{align*}
is uniformly continuous.

\end{enumerate}
\end{lemma}

\begin{definition}
Let $(S_t)_{t \geq 0}$ be a $C_0$-semigroup. The \emph{infinitesimal generator} (in short \emph{generator}) of $(S_t)_{t \geq 0}$ is the linear operator $A : X \supset \mathcal{D}(A) \rightarrow X$, which is defined on the domain
\begin{align*}
\mathcal{D}(A) := \bigg\{ x \in X : \lim_{t \rightarrow 0} \frac{S_t x - x}{t} \text{ exists} \, \bigg\},
\end{align*}
and given by
\begin{align*}
A x := \lim_{t \rightarrow 0} \frac{S_t x - x}{t}.
\end{align*}
\end{definition}

Note that the domain $\mathcal{D}(A)$ is indeed a subspace of $X$. The following result gives some properties of the infinitesimal generator of a $C_0$-semigroup. Recall that we have provided the required concepts in Definitions~\ref{def-closed} and \ref{def-dense}.

\begin{proposition}\label{prop-generator-dense}
The infinitesimal generator $A : X \supset \mathcal{D}(A) \rightarrow X$ of a $C_0$-semigroup $(S_t)_{t \geq 0}$ is densely defined and closed.
\end{proposition}

We proceed with some examples of $C_0$-semigroups and their generators.

\begin{example}
For every bounded linear operator $A \in L(X)$ the family $(e^{tA})_{t \geq 0}$ given by
\begin{align*}
e^{t A} := \sum_{n=0}^{\infty} \frac{t^n A^n}{n!}
\end{align*}
is a norm continuous semigroup with generator $A$. In particular, we have $\mathcal{D}(A) = X$.
\end{example}

\begin{example}
We consider the separable Hilbert space $X = L^2(\mathbb{R})$. Let $(S_t)_{t \geq 0}$ be the \emph{shift semigroup} defined as
\begin{align*}
S_t f := f(t + \bullet), \quad t \geq 0.
\end{align*}
Then $(S_t)_{t \geq 0}$ is a semigroup of contractions with generator $A : L^2(\mathbb{R}) \supset \mathcal{D}(A) \rightarrow L^2(\mathbb{R})$ given by
\begin{align*}
\mathcal{D}(A) &= \{ f \in L^2(\mathbb{R}) : \text{ $f$ is absolutely continuous and $f' \in L^2(\mathbb{R})$} \},
\\ Af &= f'.
\end{align*}
\end{example}

\begin{example}
On the separable Hilbert space $X = L^2(\mathbb{R}^d)$ we define the \emph{heat semigroup} $(S_t)_{t \geq 0}$ by $S_0 := {\rm Id}$ and
\begin{align*}
(S_t f)(x) := \frac{1}{(4 \pi t)^{d/2}} \int_{\mathbb{R}^d} \exp \bigg( - \frac{|x-y|^2}{4t} \bigg) f(y) dy, \quad t > 0,
\end{align*}
that is, $S_t f$ arises as the convolution of $f$ with the density of the normal distribution ${\rm N}(0,2t)$. Then $(S_t)_{t \geq 0}$ is a semigroup of contractions with generator $A : L^2(\mathbb{R}^d) \supset \mathcal{D}(A) \rightarrow L^2(\mathbb{R}^d)$ given by
\begin{align*}
\mathcal{D}(A) = W^2(\mathbb{R}^d), \quad Af = \Delta f.
\end{align*}
Here $W^2(\mathbb{R}^d)$ denotes the Sobolev space
\begin{align*}
W^2(\mathbb{R}^d) = \{ f \in L^2(\mathbb{R}^d) : D^{(\alpha)} f \in L^2(\mathbb{R}^d) \text{ exists for all $\alpha \in \mathbb{N}_0^d$ with $|\alpha| \leq 2$} \}
\end{align*}
and $\Delta$ the Laplace operator
\begin{align*}
\Delta = \sum_{i=1}^d \frac{\partial^2}{\partial x_i^2}.
\end{align*}
\end{example}

We proceed with some results regarding calculations with strongly continuous semigroups and their generators.

\begin{lemma}\label{lemma-hg-rules}
Let $(S_t)_{t \geq 0}$ be a $C_0$-semigroup with infinitesimal generator $A$. Then the following statements are true:
\begin{enumerate}
\item For every $x \in \mathcal{D}(A)$ the mapping
\begin{align*}
\mathbb{R}_+ \rightarrow X, \quad t \mapsto S_t x
\end{align*}
belongs to class $C^1(\mathbb{R}_+;X)$, and for all $t \geq 0$ we have $S_t x \in \mathcal{D}(A)$ and
\begin{align*}
\frac{d}{dt} S_t x = A S_t x = S_t A x.
\end{align*}

\item For all $x \in X$ and $t \geq 0$ we have $\int_0^t S_s x ds \in \mathcal{D}(A)$ and
\begin{align*}
A \bigg( \int_0^t S_s x \, ds \bigg) = S_t x - x.
\end{align*}

\item For all $x \in \mathcal{D}(A)$ and $t \geq 0$ we have
\begin{align*}
\int_0^t S_s A x \, ds = S_t x - x.
\end{align*}
\end{enumerate}
\end{lemma}

The following result shows that the strongly continuous semigroup $(S_t)_{t \geq 0}$ associated to some generator $A$ is unique. This explains the term \emph{generator}.

\begin{proposition}
Two $C_0$-semigroups $(S_t)_{t \geq 0}$ and $(T_t)_{t \geq 0}$ with the same infinitesimal generator $A$ coincide, that is, we have $S_t = T_t$ for all $t \geq 0$.
\end{proposition}

The next result characterizes all norm continuous semigroups in terms of their generators.

\begin{proposition}\label{prop-normstetig}
Let $(S_t)_{t \geq 0}$ be a $C_0$-semigroup with infinitesimal generator $A$. Then the following statements are equivalent:
\begin{enumerate}
\item The semigroup $(S_t)_{t \geq 0}$ is norm continuous.

\item The operator $A$ is continuous.

\item The domain of $A$ is given by $\mathcal{D}(A) = X$.
\end{enumerate}
If the previous conditions are satisfied, then we have $S_t = e^{tA}$ for all $t \geq 0$.
\end{proposition}

Now, we are interested in characterizing all linear operators $A$ which are the infinitesimal generator of some strongly continuous semigroup $(S_t)_{t \geq 0}$. The following theorem of Hille-Yosida gives a characterization in terms of the resolvent, which we have introduced in Definition~\ref{def-resolvent}.

\begin{theorem}
\emph{(Hille-Yosida theorem)} Let $A : X \supset \mathcal{D}(A) \rightarrow X$ be a linear operator and let $M \geq 1$, $\omega \in \mathbb{R}$ be constants. Then the following statements are equivalent:
\begin{enumerate}
\item $A$ is the generator of a $C_0$-semigroup $(S_t)_{t \geq 0}$ with growth estimate (\ref{Halbgruppe-Wachstum}).

\item $A$ is densely defined, closed, we have $(\omega, \infty) \subset \rho(A)$ and
\begin{align*}
\|  R(\lambda, A)^{n} \| \leq M (\lambda - \omega)^{-n} \quad \text{for all $\lambda \in (\omega, \infty)$ and $n \in \mathbb{N}$.}
\end{align*}
\end{enumerate}
\end{theorem}

In particular, we obtain the following characterization of the generators of semigroups of contractions:

\begin{corollary}
For a linear operator $A : X \supset \mathcal{D}(A) \rightarrow X$ the following statements are equivalent:
\begin{enumerate}
\item $A$ is the generator of a semigroup $(S_t)_{t \geq 0}$ of contractions.

\item $A$ is densely defined, closed, we have $(0, \infty) \subset \rho(A)$ and
\begin{align*}
\|  R(\lambda, A) \| \leq \frac{1}{\lambda} \quad \text{for all $\lambda \in (0, \infty)$.}
\end{align*}
\end{enumerate}
\end{corollary}

\begin{proposition}\label{prop-D-A2}
Let $(S_t)_{t \geq 0}$ be a $C_0$-semigroup on $X$ with generator $A$. Then the family $(S_t|_{\mathcal{D}(A)})_{t \geq 0}$ is a $C_0$-semigroup on $(\mathcal{D}(A), \| \cdot \|_{\mathcal{D}(A)})$ with generator $A : \mathcal{D}(A^2) \subset \mathcal{D}(A) \rightarrow \mathcal{D}(A^2)$, where the domain is given by
\begin{align*}
\mathcal{D}(A^2) = \{ x \in \mathcal{D}(A) : Ax \in \mathcal{D}(A) \}.
\end{align*}
\end{proposition}

Recall that we have introduced the adjoint operator for operators in Hilbert spaces in Definition~\ref{def-adjoint}.

\begin{proposition}\label{prop-hg-adj}
Let $H$ be a Hilbert space and let $(S_t)_{t \geq 0}$ be a $C_0$-semigroup on $H$ with generator $A$. Then the family of adjoint operators $(S_t^*)_{t \geq 0}$ is a $C_0$-semigroup on $H$ with generator $A^*$.
\end{proposition}

\section{Stochastic processes in infinite dimension}\label{sec-processes}

In this section, we recall the required foundations about stochastic processes in infinite dimension. In particular, we recall the definition of a trace class Wiener process and outline the construction of the It\^{o} integral.

In the sequel, $(\Omega, \mathcal{F}, (\mathcal{F}_t)_{t \geq 0}, \mathbb{P})$ denotes a filtered probability space satisfying the usual conditions.
Let $\mathbb{H}$ be a separable Hilbert space and let $Q \in L(\mathbb{H})$ be a nuclear, self-adjoint, positive definite linear operator.

\begin{definition}\label{def-Wiener-process}
A $\mathbb{H}$-valued, adapted, continuous process $W$ is called a \emph{$Q$-Wiener process}, if the following conditions are satisfied:
\begin{itemize}
\item We have $W_0 = 0$.

\item The random variable $W_t - W_s$ and the $\sigma$-algebra $\mathcal{F}_s$ are independent for all $0 \leq s \leq t$.

\item We have $W_t - W_s \sim {\rm N}(0, (t-s)Q)$ for all $0 \leq s \leq t$.
\end{itemize}
\end{definition}

In Definition~\ref{def-Wiener-process}, the distribution ${\rm N}(0, (t-s)Q)$ is a Gaussian measure with mean~$0$ and covariance operator $(t-s)Q$, see, e.g. \cite[Section~2.3.2]{Da_Prato}. The operator $Q$ is also called the covariance operator of the Wiener process $W$. As $Q$ is a trace class operator, we also call $W$ a \emph{trace class Wiener process}.

Now, let $W$ be a $Q$-Wiener process.
Then, there exist an orthonormal basis $(e_j)_{j \in \mathbb{N}}$ of $\mathbb{H}$ and a sequence $(\lambda_j)_{j \in \mathbb{N}} \subset (0,\infty)$ with $\sum_{j \in \mathbb{N}} \lambda_j < \infty$ such that
\begin{align*}
Qu = \sum_{j \in \mathbb{N}} \lambda_j \langle u,e_j \rangle_{\mathbb{H}} \, e_j, \quad u \in \mathbb{H}
\end{align*}
namely, the $\lambda_j$ are the eigenvalues of $Q$, and each $e_j$
is an eigenvector corresponding to $\lambda_j$. The space $\mathbb{H}_0 := Q^{1/2}(\mathbb{H})$, equipped with the inner product
\begin{align*}
\langle u,v \rangle_{\mathbb{H}_0} := \langle Q^{-1/2} u, Q^{-1/2} v \rangle_{\mathbb{H}},
\end{align*}
is another separable Hilbert space
and $( \sqrt{\lambda_j} e_j )_{j \in \mathbb{N}}$ is an orthonormal basis.
According to \cite[Proposition~4.1]{Da_Prato}, the sequence of stochastic
processes $( \beta^j )_{j \in \mathbb{N}}$ defined as
\begin{align}\label{Wiener-series}
\beta^j := \frac{1}{\sqrt{\lambda_j}} \langle W, e_j \rangle_{\mathbb{H}}, \quad j \in \mathbb{N}
\end{align}
is a sequence of
real-valued independent standard Wiener processes and we
have the expansion
\begin{align*}
W = \sum_{j \in \mathbb{N}} \sqrt{\lambda _j} \beta^j e_j.
\end{align*}
Now, let us briefly sketch the construction of the It\^{o} integral with respect to the Wiener process $W$. Further details can be found in \cite{Da_Prato, Atma-book}. We denote by $L_2^0(H) := L_2(\mathbb{H}_0,H)$ the space of Hilbert-Schmidt
operators from $\mathbb{H}_0$ into $H$, which, endowed with the
Hilbert-Schmidt norm
\begin{align*}
\| \Phi \|_{L_2^0(H)} := \bigg( \sum_{j \in \mathbb{N}} \lambda_j \| \Phi e_j \|^2 \bigg)^{1/2},
\quad \Phi \in L_2^0(H)
\end{align*}
itself is a separable Hilbert space. The construction of the It\^{o} integral is divided into three steps:
\begin{enumerate}
\item For every $L(\mathbb{H},H)$-valued simple process of the form
\begin{align*}
X = X_0 \mathbbm{1}_{\{0\}} + \sum_{i=1}^n X_i \mathbbm{1}_{(t_i,t_{i+1}]}
\end{align*}
with $0 = t_1 < \ldots < t_{n+1} = T$ and $\mathcal{F}_{t_i}$-measurable random variables $X_i : \Omega \rightarrow L(\mathbb{H},H)$ for $i = 1,\ldots,n$ we set
\begin{align*}
\int_0^t X_s dW_s := \sum_{i=1}^n X_i (W_{t \wedge t_{i+1}} - W_{t \wedge t_i} ).
\end{align*}
\item For every predictable $L_2^0(H)$-valued process $X$ satisfying
\begin{align*}
\mathbb{E} \bigg[ \int_0^T \| X_s \|_{L_2^0(H)}^2 ds \bigg] < \infty
\end{align*}
we extend the It\^{o} integral $\int_0^t X_s dW_s$ by an extension argument for linear operators. In particular, we obtain the \emph{It\^{o} isometry}
\begin{align}\label{Ito-isom}
\mathbb{E} \Bigg[ \bigg\| \int_0^T X_s dW_s \bigg\|^2 \Bigg] = \mathbb{E} \bigg[ \int_0^T \| X_s \|_{L_2^0(H)}^2 ds \bigg].
\end{align}
\item By localization, we extend the It\^{o} integral $\int_0^t X_s dW_s$ for every predictable $L_2^0(H)$-valued process $X$ satisfying
\begin{align*}
\mathbb{P} \bigg( \int_0^t \| \Phi_s \|_{L_2^0(H)}^2 ds < \infty
\bigg) = 1 \quad \text{for all $t \geq 0$.}
\end{align*}
\end{enumerate}
The It\^{o} integral $(\int_0^t X_s dW_s)_{t \geq 0}$ is an $H$-valued, continuous, local martingale, and we have the series expansion 
\begin{align}\label{series-integral}
\int_0^t X_s dW_s = \sum_{j \in \mathbb{N}} \int_0^t X_s^j d\beta_s^j, \quad t \geq 0,
\end{align}
where $X^j := \sqrt{\lambda_j} X e_j$ for each $j \in \mathbb{N}$. An indispensable tool for stochastic calculus in infinite dimensions is It\^{o}'s formula, which we shall recall here.

\begin{theorem}[It\^{o}'s formula]\label{thm-Ito}
Let $E$ be another separable Hilbert space, let $f \in C_b^{1,2,{\rm loc}}(\mathbb{R}_+ \times H; E)$ be a function and let $X$ be an $H$-valued It\^{o} process of the form
\begin{align*}
X_t = X_0 + \int_0^t Y_s ds + \int_0^t Z_s dW_s, \quad t \geq 0.
\end{align*}
Then $(f(t, X_t))_{t \geq 0}$ is an $E$-valued It\^{o} process, and we have $\mathbb{P}$--almost surely
\begin{align*}
f(t, X_t) &= f (0, X_0) + \int_0^t \bigg( D_s f(s, X_s) + D_x f(s, X_s) Y_s 
\\& \quad + \frac{1}{2} \sum_{j \in \mathbb{N}} D_{xx} f(s, X_s) (Z_s^j, Z_s^j) \bigg) ds
+ \int_0^t D_x f(s, X_s) Z_s dW_s, \quad t \geq 0,
\end{align*}
where we use the notation $Z^j := \sqrt{\lambda_j} Z e_j$ for each $j \in \mathbb{N}$.
\end{theorem}

\begin{proof}
This result is a consequence of \cite[Theorem~2.9]{Atma-book}.
\end{proof}

\section{Solution concepts for SPDEs}\label{sec-solution-concepts}

In this section, we present the concepts of strong, mild and weak solutions to SPDEs and discuss their relations.

Let $H$ be a separable Hilbert space and let $(S_t)_{t \geq 0}$ be a $C_0$-semigroup on $H$ with infinitesimal generator $A$. Furthermore, let $W$ be a trace class Wiener process on some separable Hilbert space $\mathbb{H}$. We consider the SPDE
\begin{align}\label{SPDE}
\left\{
\begin{array}{rcl}
dX_t & = & (AX_t + \alpha(t, X_t))dt + \sigma(t, X_t) dW_t
\medskip
\\ X_0 & = & h_0.
\end{array}
\right.
\end{align}
Here $\alpha : \mathbb{R}_+ \times H \rightarrow H$ and $\sigma : \mathbb{R}_+ \times H \rightarrow L_2^0(H)$ are measurable mappings.

\begin{definition}\label{def-solutions}
Let $h_0 : \Omega \rightarrow H$ be a $\mathcal{F}_0$-measurable random variable and let $\tau > 0$ be a strictly positive stopping time. Furthermore, let $X = X^{(h_0)}$ be an $H$-valued, continuous, adapted process such that
\begin{align*}
\mathbb{P} \bigg( \int_{0}^{t \wedge \tau} \big( \| X_s \| + \| \alpha(s, X_s) \| + \| \sigma(s, X_s) \|_{L_2^0(H)}^2 \big) ds < \infty \bigg) = 1 \quad \text{for all $t \geq 0$.}
\end{align*}

\begin{enumerate}
\item $X$ is called a \emph{local strong solution} to (\ref{SPDE}), if
\begin{align}\label{X-in-domain}
&X_{t \wedge \tau} \in \mathcal{D}(A) \quad \text{for all $t \geq 0$,} \quad \text{$\mathbb{P}$--almost surely,}
\\ \label{AX-int} &\mathbb{P} \bigg( \int_{0}^{t \wedge \tau} \| A X_{s} \| ds < \infty \bigg) = 1 \quad \text{ for all } t \geq 0
\end{align}
and $\mathbb{P}$--almost surely we have
\begin{align*}
X_{t \wedge \tau} &= h_0 + \int_{0}^{t \wedge \tau} \big( A X_s + \alpha(s, X_s) \big) ds + \int_{0}^{t \wedge \tau} \sigma(s, X_s) dW_s, \quad t \geq 0.
\end{align*}
\item $X$ is called a \emph{local weak solution} to (\ref{SPDE}), if for all $\zeta \in \mathcal{D}(A^*)$ the following equation is fulfilled $\mathbb{P}$--almost surely:
\begin{align*}
\langle \zeta, X_{t \wedge \tau} \rangle &= \langle \zeta, h_0 \rangle + \int_{0}^{t \wedge \tau} \big( \langle A^* \zeta, X_s \rangle + \langle \zeta, \alpha(s, X_s) \rangle \big) ds 
\\ &\quad + \int_{0}^{t \wedge \tau} \langle \zeta, \sigma(s, X_s) \rangle dW_s, \quad t \geq 0.
\end{align*}
\item $X$ is called a \emph{local mild solution} to (\ref{SPDE}), if $\mathbb{P}$--almost surely we have
\begin{align*}
X_{t \wedge \tau} &= S_{t \wedge \tau} h_0 + \int_{0}^{t \wedge \tau} S_{(t \wedge \tau) - s} \alpha(s, X_s) ds + \int_{0}^{t \wedge \tau} S_{(t \wedge \tau) - s} \sigma(s, X_s) dW_s, \quad t \geq 0.
\end{align*}
\end{enumerate}
We call $\tau$ the \emph{lifetime} of $X$. If $\tau \equiv \infty$, then we call $X$ a \emph{strong}, \emph{weak} or \emph{mild solution} to (\ref{SPDE}), respectively.
\end{definition}

\begin{remark}
Note that the concept of a strong solution is rather restrictive, because condition (\ref{X-in-domain}) has to be fulfilled.
\end{remark}

For what follows, we fix a $\mathcal{F}_0$-measurable random variable $h_0 : \Omega \rightarrow H$ and a strictly positive stopping time $\tau > 0$.

\begin{proposition}\label{prop-stark-schwach}
Every local strong solution $X$ to (\ref{SPDE}) with lifetime $\tau$ is also a local weak solution to (\ref{SPDE}) with lifetime $\tau$.
\end{proposition}

\begin{proof}
Let $X$ be a local strong solution to (\ref{SPDE}) with lifetime $\tau$. Furthermore, let $\zeta \in \mathcal{D}(A^*)$ be arbitrary. Then we have $\mathbb{P}$--almost surely for all $t \geq 0$ the identities
\begin{align*}
\langle \zeta, X_{t \wedge \tau} \rangle &= \Big\langle \zeta, h_0 + \int_{0}^{t \wedge \tau} \big( A X_s + \alpha(s, X_s) \big) ds + \int_{0}^{t \wedge \tau} \sigma(s, X_s) dW_s \Big\rangle
\\ &= \langle \zeta, h_0 \rangle + \int_{0}^{t \wedge \tau} \langle \zeta, AX_s + \alpha(s, X_s) \rangle ds + \int_{0}^{t \wedge \tau} \langle \zeta, \sigma(s, X_s) \rangle dW_s
\\ &= \langle \zeta, h_0 \rangle + \int_{0}^{t \wedge \tau} \big( \langle A^* \zeta, X_s \rangle + \langle \zeta, \alpha(s, X_s) \rangle \big) ds + \int_{0}^{t \wedge \tau} \langle \zeta, \sigma(s, X_s) \rangle dW_s,
\end{align*}
showing that $X$ is also a local weak solution to (\ref{SPDE}) with lifetime $\tau$.
\end{proof}

\begin{proposition}\label{prop-strak-schwach-eq}
Let $X$ be a stochastic process with $X_0 = h_0$. Then the following statements are equivalent:
\begin{enumerate}
\item The process $X$ is a local strong solution to (\ref{SPDE}) with lifetime $\tau$.

\item The process $X$ is a local weak solution to (\ref{SPDE}) with lifetime $\tau$, and we have (\ref{X-in-domain}), (\ref{AX-int}).
\end{enumerate}
\end{proposition}

\begin{proof}
(1) $\Rightarrow$ (2): This implication is a direct consequence of Proposition~\ref{prop-stark-schwach}.

\noindent(2) $\Rightarrow$ (1): Let $\zeta \in \mathcal{D}(A^*)$ be arbitrary. Then we have $\mathbb{P}$--almost surely for all $t \geq 0$ the identities
\begin{align*}
\langle \zeta, X_{t \wedge \tau} \rangle &= \langle \zeta, h_0 \rangle + \int_{0}^{t \wedge \tau} \big( \langle A^* \zeta, X_s \rangle + \langle \zeta, \alpha(s, X_s) \rangle \big) ds + \int_{0}^{t \wedge \tau} \langle \zeta, \sigma(s, X_s) \rangle dW_s
\\ &= \langle \zeta, h_0 \rangle + \int_{0}^{t \wedge \tau} \langle \zeta, AX_s + \alpha(s, X_s) \rangle ds + \int_{0}^{t \wedge \tau} \langle \zeta, \sigma(s, X_s) \rangle dW_s
\\ &= \Big\langle \zeta, h_0 + \int_{0}^{t \wedge \tau} \big( A X_s + \alpha(s, X_s) \big) ds + \int_{0}^{t \wedge \tau} \sigma(s, X_s) dW_s \Big\rangle.
\end{align*}
By Proposition~\ref{prop-A-adj-dicht} the domain $\mathcal{D}(A^*)$ is dense in $H$, and hence we obtain $\mathbb{P}$--almost surely
\begin{align*}
X_{t \wedge \tau} &= h_0 + \int_{0}^{t \wedge \tau} \big( A X_s + \alpha(s, X_s) \big) ds + \int_{0}^{t \wedge \tau} \sigma(s, X_s) dW_s, \quad t \geq 0.
\end{align*}
Consequently, the process $X$ is also a local strong solution to (\ref{SPDE}) with lifetime $\tau$.
\end{proof}

\begin{corollary}\label{cor-strong-sol-manifold}
Let $\mathcal{M} \subset \mathcal{D}(A)$ be a subset such that $A$ is continuous on $\mathcal{M}$, and let $X$ be a local weak solution to (\ref{SPDE}) with lifetime $\tau$ such that
\begin{align}\label{X-in-M}
X_{t \wedge \tau} \in \mathcal{M} \quad \text{for all $t \geq 0$,} \quad \text{$\mathbb{P}$--almost surely.}
\end{align}
Then $X$ is also a local strong solution to (\ref{SPDE}) with lifetime $\tau$.
\end{corollary}

\begin{proof}
Since $\mathcal{M} \subset \mathcal{D}(A)$, condition (\ref{X-in-M}) implies that (\ref{X-in-domain}) is fulfilled. Moreover, by the continuity of $A$ on $\mathcal{M}$, the sample paths of the process $A X$ are $\mathbb{P}$--almost surely continuous, and hence, we obtain (\ref{AX-int}). Consequently, using Proposition~\ref{prop-strak-schwach-eq}, the process $X$ is also a local strong solution to (\ref{SPDE}) with lifetime $\tau$.
\end{proof}

\begin{proposition}\label{prop-stark-mild}
Every strong solution $X$ to (\ref{SPDE}) is also a mild solution to (\ref{SPDE}).
\end{proposition}

\begin{proof}
According to Lemma~\ref{lemma-domain}, the domain $(\mathcal{D}(A), \| \cdot \|_{\mathcal{D}(A)})$ endowed with the graph norm is a separable Hilbert space, too. Hence, by Lemma~\ref{lemma-hg-rules}, for all $t \geq 0$ the function
\begin{align*}
f : [0, t] \times \mathcal{D}(A) \rightarrow H, \quad f(s, x) := S_{t-s} x.
\end{align*}
belongs to the class $C_b^{1,2,{\rm loc}}([0, t] \times \mathcal{D}(A); H)$ with partial derivatives
\begin{align*}
D_t f(t, x) &= -A S_{t-s} x,
\\ D_x f(t, x) &= S_{t-s},
\\ D_{xx} f(t, x) &= 0.
\end{align*}
Hence, by It\^{o}'s formula (see Theorem~\ref{thm-Ito}) and Lemma~\ref{lemma-hg-rules} we obtain $\mathbb{P}$--almost surely
\begin{align*}
X_t &= f(t, X_t) = f(0, h_0) + \int_0^t \big( D_s f(s, X_s) + D_x f(s, X_s) (A X_s + \alpha(s, X_s) ) \big) ds
\\ &\quad + \int_0^t D_x f(s, X_s) \sigma(s, X_s) dW_s
\\ &= S_t h_0 + \int_0^t \big( -A S_{t-s} X_s + S_{t-s} ( A X_s + \alpha(s, X_s) ) \big) ds + \int_0^t S_{t-s} \sigma(s, X_s) dW_s
\\ &= S_{t} h_0 + \int_{0}^{t} S_{t - s} \alpha(s, X_s) ds + \int_{0}^{t} S_{t - s} \sigma(s, X_s) dW_s.
\end{align*} 
Thus, $X$ is also a mild solution to (\ref{SPDE}).
\end{proof}

We recall the following technical auxiliary result without proof and refer, e.g., to \cite[Section~3.1]{Atma-book}.

\begin{lemma}\label{lemma-dicht-C1-D-adj}
Let $T \geq 0$ be arbitrary. Then the linear space
\begin{align*}
U_T := {\rm lin} \, \{ g \zeta : g \in C^1([0, T]; \mathbb{R}) \text{ and } \zeta \in \mathcal{D}(A^*) \}
\end{align*}
is dense in $C^1([0, T], \mathcal{D}(A^*))$, where $(\mathcal{D}(A^*), \| \cdot \|_{\mathcal{D}(A^*)})$ is endowed with the graph norm.
\end{lemma}

\begin{lemma}\label{lemma-id-schwach-mild}
Let $X$ be a weak solution to (\ref{SPDE}). Then for all $T \geq 0$ and all $f \in C^1([0, T], \mathcal{D}(A^*))$ we have $\mathbb{P}$--almost surely
\begin{equation}\label{Ito-formula-mild}
\begin{aligned}
\langle f(t), X_{t} \rangle &= \langle f(0), h_0 \rangle
+ \int_0^{t} \big( \langle f'(s) + A^* f(s), X_s \rangle
+ \langle f(s), \alpha(s, X_s) \rangle \big) ds
\\ & \quad + \int_0^{t} \langle f(s), \sigma(s, X_s) \rangle d
W_s, \quad t \in [0,T].
\end{aligned}
\end{equation}
\end{lemma}

\begin{proof}
By virtue of Lemma~\ref{lemma-dicht-C1-D-adj}, it suffices to prove formula (\ref{Ito-formula-mild}) for all $f \in U_T$. Let $f \in U_T$ be arbitrary. Then there are $g_1,\ldots,g_n \in C^1([0, T]; \mathbb{R})$ and $\zeta_1,\ldots,\zeta_n \in \mathcal{D}(A^*)$ for some $n \in \mathbb{N}$ such that
\begin{align*}
f(t) = \sum_{i=1}^n g_i(t) \zeta_i, \quad t \in [0, T].
\end{align*}
We define the function
\begin{align*}
F : [0, T] \times \mathbb{R}^n \rightarrow \mathbb{R}, \quad F(t, x) := \sum_{i=1}^n g_i(t) x_i. 
\end{align*}
Then we have $F \in C^{1,2}([0, T] \times \mathbb{R}^n; \mathbb{R})$ with partial derivatives
\begin{align*}
D_t F(t, x) &= \sum_{i=1}^n g_i'(t) x_i,
\\ D_x F(t, x) &= \langle g(t), \bullet \rangle_{\mathbb{R}^n},
\\ D_{xx} F(t, x) &= 0.
\end{align*}
Since $X$ is a weak solution to (\ref{SPDE}), the $\mathbb{R}^n$-valued process
\begin{align*}
\langle \zeta, X \rangle := \langle \zeta_i, X \rangle_{i=1,\ldots,n} 
\end{align*}
is an It\^{o} process with representation
\begin{align*}
\langle \zeta, X_{t} \rangle &= \langle \zeta, h_0 \rangle + \int_{0}^{t} \big( \langle A^* \zeta, X_s \rangle + \langle \zeta, \alpha(s, X_s) \rangle \big) ds + \int_{0}^{t} \langle \zeta, \sigma(s, X_s) \rangle dW_s, \quad t \geq 0.
\end{align*}
By It\^{o}'s formula (Theorem~\ref{thm-Ito}) we obtain $\mathbb{P}$--almost surely
\begin{align*}
&\langle f(t), X_t \rangle = \Big\langle \sum_{i=1}^n g_i(t) \zeta_i, X_t \Big\rangle = \sum_{i=1}^n g_i(t) \langle \zeta_i, X_t \rangle = F(t, \langle \zeta, X_t \rangle)
\\ &= F(0, \langle \zeta, h_0 \rangle ) 
\\ &\quad + \int_0^t \big( D_s F(s, \langle \zeta, X_s \rangle ) + D_x F(s, \langle \zeta, X_s \rangle) \big( \langle A^* \zeta, X_s \rangle + \langle \zeta, \alpha(s, X_s) \rangle \big) \big) ds
\\ &\quad + \int_0^t D_x F(s, \langle \zeta, X_s \rangle) \langle \zeta, \sigma(s, X_s) \rangle dW_s
\\ &= \sum_{i=1}^n g_i(0) \langle \zeta_i, h_0 \rangle 
\\ &\quad + \int_0^t \bigg( \sum_{i=1}^n g_i'(t) \langle \zeta_i, X_s \rangle + \sum_{i=1}^n g_i(t) \big( \langle A^* \zeta_i, X_s \rangle + \langle \zeta_i, \alpha(s, X_s) \rangle \big) \bigg) ds
\\ &\quad + \int_0^t \bigg( \sum_{i=1}^n g_i(s) \langle \zeta_i, \sigma(s, X_s) \rangle \bigg) dW_s \quad t \in [0, T],
\end{align*}
and hence
\begin{align*}
&\langle f(t), X_t \rangle = \Big\langle \sum_{i=1}^n g_i(0) \zeta_i, h_0 \Big\rangle 
\\ &\quad + \int_0^t \bigg( \Big\langle \sum_{i=1}^n g_i'(s) \zeta_i, X_s \Big\rangle + \Big\langle A^* \Big( \sum_{i=1}^n g_i(s) \zeta_i \Big), X_s \Big\rangle + \Big\langle \sum_{i=1}^n g_i(s) \zeta_i, \alpha(s, X_s) \Big\rangle \bigg) ds
\\ &\quad + \int_0^t \Big\langle \sum_{i=1}^n g_i(s) \zeta_i, \sigma(s, X_s) \Big\rangle dW_s
\\ &= \langle f(0), h_0 \rangle
+ \int_0^{t} \Big( \langle f'(s) + A^* f(s), X_s \rangle
+ \langle f(s), \alpha(s, X_s) \rangle \Big) ds
\\ & \quad + \int_0^{t} \langle f(s), \sigma(s, X_s) \rangle d
W_s, \quad t \in [0, T].
\end{align*}
This concludes the proof.
\end{proof}

\begin{proposition}\label{prop-schwach-mild}
Every weak solution $X$ to (\ref{SPDE}) is also a mild solution to (\ref{SPDE}).
\end{proposition}

\begin{proof}
By Proposition~\ref{prop-hg-adj}, the family $(S_t^*)_{t \geq 0}$ is a $C_0$-semigroup with generator $A^*$. Thus, Proposition~\ref{prop-D-A2} yields that the family of restrictions $(S_t^*|_{\mathcal{D}(A^*)})_{t \geq 0}$ is a $C_0$-semigroup on $(\mathcal{D}(A^*), \| \cdot \|_{\mathcal{D}(A^*)})$ with generator $A^* : \mathcal{D}((A^*)^2) \subset \mathcal{D}(A^*) \rightarrow \mathcal{D}((A^*)^2)$. 

Now, let $t \geq 0$ and $\zeta \in \mathcal{D}((A^*)^2)$ be arbitrary. We define the function
\begin{align*}
f : [0, t] \rightarrow \mathcal{D}(A^*), \quad f(s) := S_{t-s}^* \zeta.
\end{align*}
By Lemma~\ref{lemma-hg-rules} we have $f \in C^1([0, t]; \mathcal{D}(A^*))$ with derivative
\begin{align*}
f'(s) = - A^* S_{t-s}^* \zeta = -A^* f(s).
\end{align*}
Using Lemma~\ref{lemma-id-schwach-mild}, we obtain $\mathbb{P}$--almost surely
\begin{align*}
\langle \zeta, X_t \rangle &= \langle f(t), X_t \rangle
\\ &= \langle f(0), h_0 \rangle
+ \int_0^{t} \langle f(s), \alpha(s, X_s) \rangle ds
+ \int_0^{t} \langle f(s), \sigma(s, X_s) \rangle d
W_s
\\ &= \langle S_t^* \zeta, h_0 \rangle
+ \int_0^{t} \langle S_{t-s}^* \zeta, \alpha(s, X_s) \rangle ds
+ \int_0^{t} \langle S_{t-s}^* \zeta, \sigma(s, X_s) \rangle d
W_s
\\ &= \langle \zeta, S_t h_0 \rangle
+ \int_0^{t} \langle \zeta, S_{t-s} \alpha(s, X_s) \rangle ds
+ \int_0^{t} \langle \zeta, S_{t-s} \sigma(s, X_s) \rangle d
W_s
\\ &= \Big\langle \zeta, S_t h_0 + \int_0^{t} S_{t-s} \alpha(s, X_s) ds
+ \int_0^{t} S_{t-s} \sigma(s, X_s) dW_s \Big\rangle.
\end{align*}
Since, by Proposition~\ref{prop-generator-dense}, the domain $\mathcal{D}((A^*)^2)$ is dense in $( \mathcal{D}(A^*), \| \cdot \|_{\mathcal{D}(A^*)} )$, we get $\mathbb{P}$--almost surely for all $\zeta \in \mathcal{D}(A^*)$ the identity
\begin{align*}
\langle \zeta, X_t \rangle = \Big\langle \zeta, S_t h_0 + \int_0^{t} S_{t-s} \alpha(s, X_s) ds
+ \int_0^{t} S_{t-s} \sigma(s, X_s) dW_s \Big\rangle.
\end{align*}
Since, by Proposition~\ref{prop-generator-dense}, the domain $\mathcal{D}(A^*)$ is dense in $H$, we obtain $\mathbb{P}$--almost surely
\begin{align*}
X_t = S_t h_0 + \int_0^{t} S_{t-s} \alpha(s, X_s) ds
+ \int_0^{t} S_{t-s} \sigma(s, X_s) dW_s,
\end{align*}
proving that $X$ is a mild solution to (\ref{SPDE}).
\end{proof}

\begin{remark}
Now, the proof of Proposition~\ref{prop-stark-mild} is an immediate consequence of Propositions~\ref{prop-stark-schwach} and \ref{prop-schwach-mild}.
\end{remark}

We have just seen that every weak solution to (\ref{SPDE}) is also a mild solution. Under the following regularity condition (\ref{Fubini-condition}), the converse of this statement holds true as well.

\begin{proposition}\label{prop-mild-schwach}
Let $X$ be a mild solution to (\ref{SPDE}) such that
\begin{align}\label{Fubini-condition}
\mathbb{E} \bigg[ \int_0^T \| \sigma(s, X_s) \|_{L_2^0(H)}^2 ds \bigg] < \infty \quad \text{for all $T \geq 0$.}
\end{align}
Then $X$ is also a weak solution to (\ref{SPDE}).
\end{proposition}

\begin{proof}
Let $t \geq 0$ and $\zeta \in \mathcal{D}(A^*)$ be arbitrary. Using Lemma~\ref{lemma-hg-rules}, we obtain $\mathbb{P}$--almost surely
\begin{align*}
\int_0^t \langle A^* \zeta, S_s h_0 \rangle ds &= \Big\langle A^* \zeta, \underbrace{\int_0^t S_s h_0 ds}_{\in \mathcal{D}(A)} \Big\rangle = \Big\langle \zeta, A \bigg( \int_0^t S_s h_0 ds \bigg) \Big\rangle = \langle \zeta, S_t h_0 - h_0 \rangle
\\ &= \langle \zeta, S_t h_0 \rangle - \langle \zeta, h_0 \rangle.
\end{align*}
By Fubini's theorem for Bochner integrals (see \cite[Section~1.1, page 21]{Atma-book}) and Lemma~\ref{lemma-hg-rules} we obtain $\mathbb{P}$--almost surely
\begin{align*}
&\int_0^t \Big\langle A^* \zeta, \int_0^s S_{s-u} \alpha(u, X_u) du \Big\rangle ds = \Big\langle A^* \zeta, \int_0^t \bigg( \int_0^s S_{s-u} \alpha(u, X_u) du \bigg) ds \Big\rangle
\\ &= \Big\langle A^* \zeta, \int_0^t \bigg( \int_u^t S_{s-u} \alpha(u, X_u) ds \bigg) du \Big\rangle = \int_0^t \Big\langle A^* \zeta, \int_u^t S_{s-u} \alpha(u, X_u) ds \Big\rangle du
\\ &= \int_0^t \Big\langle A^* \zeta, \underbrace{\int_0^{t-s} S_{u} \alpha(s, X_s) du}_{\in \mathcal{D}(A)} \Big\rangle ds = \int_0^t \Big\langle \zeta, A \bigg( \int_0^{t-s} S_{u} \alpha(s, X_s) du \bigg) \Big\rangle ds
\\ &= \int_0^t \langle \zeta, S_{t-s} \alpha(s, X_s) - \alpha(s, X_s) \rangle ds 
\\ &= \Big\langle \zeta, \int_0^t S_{t-s} \alpha(s, X_s) ds \Big\rangle - \int_0^t \langle \zeta, \alpha(s, X_s) \rangle ds.
\end{align*}
Due to assumption (\ref{Fubini-condition}), we may use Fubini's theorem for stochastic integrals (see \cite[Theorem~2.8]{Atma-book}), which, together with Lemma~\ref{lemma-hg-rules} gives us $\mathbb{P}$--almost surely
\begin{align*}
&\int_0^t \Big\langle A^* \zeta, \int_0^s S_{s-u} \sigma(u, X_u) dW_u \Big\rangle ds = \Big\langle A^* \zeta, \int_0^t \bigg( \int_0^s S_{s-u} \sigma(u, X_u) dW_u \bigg) ds \Big\rangle
\\ &= \Big\langle A^* \zeta, \int_0^t \bigg( \int_u^t S_{s-u} \sigma(u, X_u) ds \bigg) dW_u \Big\rangle = \int_0^t \Big\langle A^* \zeta, \int_u^t S_{s-u} \sigma(u, X_u) ds \Big\rangle dW_u
\\ &= \int_0^t \Big\langle A^* \zeta, \underbrace{\int_0^{t-s} S_{u} \sigma(s, X_s) du}_{\in \mathcal{D}(A)} \Big\rangle dW_s = \int_0^t \Big\langle \zeta, A \bigg( \int_0^{t-s} S_{u} \sigma(s, X_s) du \bigg) \Big\rangle dW_s
\\ &= \int_0^t \langle \zeta, S_{t-s} \sigma(s, X_s) - \sigma(s, X_s) \rangle dW_s 
\\ &= \Big\langle \zeta, \int_0^t S_{t-s} \sigma(s, X_s) dW_s \Big\rangle - \int_0^t \langle \zeta, \sigma(s, X_s) \rangle dW_s.
\end{align*}
Therefore, and since $X$ is a mild solution to (\ref{SPDE}), we obtain $\mathbb{P}$--almost surely
\begin{align*}
\langle \zeta, X_t \rangle &= \Big\langle \zeta, S_t h_0 + \int_0^{t} S_{t-s} \alpha(s, X_s) ds
+ \int_0^{t} S_{t-s} \sigma(s, X_s) dW_s \Big\rangle
\\ &= \langle \zeta, S_t h_0 \rangle + \Big\langle \zeta, \int_0^t S_{t-s} \alpha(s, X_s) ds \Big\rangle + \Big\langle \zeta, \int_0^t S_{t-s} \sigma(s, X_s) dW_s \Big\rangle
\\ &= \langle \zeta, h_0 \rangle + \int_0^t \langle A^* \zeta, S_s h_0 \rangle ds
\\ &\quad + \int_0^t \Big\langle A^* \zeta, \int_0^s S_{s-u} \alpha(u, X_u) du \Big\rangle ds + \int_0^t \langle \zeta, \alpha(s, X_s) \rangle ds
\\ &\quad + \int_0^t \Big\langle A^* \zeta, \int_0^s S_{s-u} \sigma(u, X_u) dW_u \Big\rangle ds + \int_0^t \langle \zeta, \sigma(s, X_s) \rangle dW_s,
\end{align*}
and hence
\begin{align*}
&\langle \zeta, X_t \rangle = \langle \zeta, h_0 \rangle 
\\ &\quad + \int_0^t \Big\langle A^* \zeta, \underbrace{S_s h_0 + \int_0^s S_{s-u} \alpha(u, X_u) du + \int_0^s S_{s-u} \sigma(u, X_u) dW_u}_{= X_s} \Big\rangle ds
\\ &\quad + \int_0^t \langle \zeta, \alpha(s, X_s) \rangle ds + \int_0^t \langle \zeta, \sigma(s, X_s) \rangle dW_s
\\ &= \langle \zeta, h_0 \rangle + \int_{0}^{t} \big( \langle A^* \zeta, X_s \rangle + \langle \zeta, \alpha(s, X_s) \rangle \big) ds + \int_{0}^{t} \langle \zeta, \sigma(s, X_s) \rangle dW_s.
\end{align*}
Consequently, the process $X$ is also a weak solution to (\ref{SPDE}).
\end{proof}

Next, we provide conditions which ensure that a mild solution to (\ref{SPDE}) is also a strong solution.

\begin{proposition}
Let $X$ be a mild solution to (\ref{SPDE}) such that $\mathbb{P}$--almost surely we have
\begin{align}\label{cond-dom-1}
&X_s, \alpha(s, X_s) \in \mathcal{D}(A) \text{ and } \sigma(s, X_s) \in L_2^0(\mathcal{D}(A)) \quad \text{for all $s \geq 0$,}
\end{align}
as well as
\begin{align}\label{cond-dom-2}
&\mathbb{P} \bigg( \int_0^t \big( \| X_s \|_{\mathcal{D}(A)} + \| \alpha(s, X_s) \|_{\mathcal{D}(A)} \big) ds < \infty \bigg) = 1 \quad \text{for all $t \geq 0$,}
\\ \label{Fubini-condition-2} &\mathbb{E} \bigg[ \int_0^T \| \sigma(s, X_s) \|_{L_2^0(\mathcal{D}(A))}^2 ds \bigg] < \infty \quad \text{for all $T \geq 0$.}
\end{align}
Then $X$ is also a strong solution to (\ref{SPDE}).
\end{proposition}

\begin{proof}
By hypotheses (\ref{cond-dom-1}) and (\ref{cond-dom-2}) we have (\ref{X-in-domain}) and (\ref{AX-int}). Let $t \geq 0$ be arbitrary. By Lemma~\ref{lemma-hg-rules} we have
\begin{align*}
S_t h_0 - h_0 = \int_0^t A S_s h_0 ds.
\end{align*}
Furthermore, by Lemma~\ref{lemma-hg-rules} and Fubini's theorem for Bochner integrals (see \cite[Section~1.1, page 21]{Atma-book}) we have $\mathbb{P}$--almost surely
\begin{align*}
&\int_0^t \big( S_{t-s} \alpha(s, X_s) - \alpha(s, X_s) \big) ds = \int_0^t \bigg( \int_0^{t-s} A S_u \alpha(s, X_s) du \bigg) ds
\\ &= \int_0^t \bigg( \int_u^{t} A S_{s-u} \alpha(u, X_u) ds \bigg) du = \int_0^t \bigg( \int_0^s A S_{s-u} \alpha(u, X_u) du \bigg) ds
\\ &= \int_0^t A \bigg( \int_0^s S_{s-u} \alpha(u, X_u) du \bigg) ds.
\end{align*}
Due to assumption (\ref{Fubini-condition-2}), we may use Fubini's theorem for stochastic integrals (see \cite[Theorem~2.8]{Atma-book}), which, together with Lemma~\ref{lemma-hg-rules} gives us $\mathbb{P}$--almost surely
\begin{align*}
&\int_0^t \big( S_{t-s} \sigma(s, X_s) - \sigma(s, X_s) \big) dW_s = \int_0^t \bigg( \int_0^{t-s} A S_u \sigma(s, X_s) du \bigg) dW_s
\\ &= \int_0^t \bigg( \int_u^{t} A S_{s-u} \sigma(u, X_u) ds \bigg) dW_u = \int_0^t \bigg( \int_0^s A S_{s-u} \sigma(u, X_u) dW_u \bigg) ds
\\ &= \int_0^t A \bigg( \int_0^s S_{s-u} \sigma(u, X_u) dW_u \bigg) ds. 
\end{align*}
Since $X$ is a mild solution to (\ref{SPDE}), we have $\mathbb{P}$--almost surely
\begin{align*}
X_t &= S_t h_0 + \int_0^{t} S_{t-s} \alpha(s, X_s) ds
+ \int_0^{t} S_{t-s} \sigma(s, X_s) dW_s
\\ &= h_0 + \int_0^t \alpha(s, X_s) ds + \int_0^t \sigma(s, X_s) dW_s 
\\ &\quad + (S_t h_0 - h_0) + \int_0^t (S_{t-s} \alpha(s, X_s) - \alpha(s, X_s)) ds 
\\ &\quad + \int_0^t (S_{t-s} \sigma(s, X_s) - \sigma(s, X_s)) dW_s,
\end{align*}
and hence, combining the latter identities, we obtain $\mathbb{P}$--almost surely
\begin{align*}
X_t &= h_0 + \int_0^t \alpha(s, X_s) ds + \int_0^t \sigma(s, X_s) dW_s 
\\ &\quad + \int_0^t A S_s h_0 ds + \int_0^t A \bigg( \int_0^s S_{s-u} \alpha(u, X_u) du \bigg) ds 
\\ &\quad + \int_0^t A \bigg( \int_0^s S_{s-u} \sigma(u, X_u) dW_u \bigg) ds,
\end{align*}
which implies
\begin{align*}
X_t &= h_0 + \int_0^t \alpha(s, X_s) ds + \int_0^t \sigma(s, X_s) dW_s 
\\ &\quad + \int_0^t A \underbrace{\bigg( S_s h_0 + \int_0^s S_{s-u} \alpha(u, X_u) du + \int_0^s S_{s-u} \sigma(u, X_u) dW_u \bigg)}_{= X_s} ds 
\\ &= h_0 + \int_{0}^{t} \big( A X_s + \alpha(s, X_s) \big) ds + \int_{0}^{t} \sigma(s, X_s) dW_s.
\end{align*}
This proves that $X$ is also a strong solution to (\ref{SPDE}).
\end{proof}

The following result shows that for norm continuous semigroups the concepts of strong, weak and mild solutions are equivalent. In particular, this applies for finite dimensional state spaces. 

\begin{proposition}\label{prop-SPDE-normstetig}
Suppose the semigroup $(S_t)_{t \geq 0}$ is norm continuous. Let $X$ be a stochastic process with $X_0 = h_0$. Then the following statements are equivalent:
\begin{enumerate}
\item The process $X$ is a strong solution to (\ref{SPDE}).

\item The process $X$ is a weak solution to (\ref{SPDE}).

\item The process $X$ is a mild solution to (\ref{SPDE}).
\end{enumerate}
\end{proposition}

\begin{proof}
(1) $\Rightarrow$ (2): This implication is a consequence of Proposition~\ref{prop-stark-schwach}.

\noindent(2) $\Rightarrow$ (3): This implication is a consequence of Proposition~\ref{prop-schwach-mild}.

\noindent(3) $\Rightarrow$ (1): By Proposition~\ref{prop-normstetig} we have $A \in L(H)$ and $S_t = e^{tA}$, $t \geq 0$. Furthermore, the family $(e^{tA})_{t \in \mathbb{R}}$ is a $C_0$-group on $H$. Therefore, and since $X$ is a mild solution to (\ref{SPDE}), we have $\mathbb{P}$--almost surely
\begin{align*}
X_t &= e^{tA} h_0 + \int_0^{t} e^{(t-s)A} \alpha(s, X_s) ds
+ \int_0^{t} e^{(t-s)A} \sigma(s, X_s) dW_s
\\ &= e^{tA} h_0 + e^{tA} \int_0^{t} e^{-s A} \alpha(s, X_s) ds
+ e^{tA} \int_0^{t} e^{-s A} \sigma(s, X_s) dW_s, \quad t \geq 0.
\end{align*}
Let $Y$ be the It\^{o} process
\begin{align*}
Y_t := \int_0^{t} e^{-s A} \alpha(s, X_s) ds
+ \int_0^{t} e^{-s A} \sigma(s, X_s) dW_s, \quad t \geq 0.
\end{align*}
Then we have $\mathbb{P}$--almost surely
\begin{align*}
X_t = e^{tA} (h_0 + Y_t), \quad t \geq 0,
\end{align*}
and, by Lemma~\ref{lemma-hg-rules}, we have
\begin{align*}
e^{t A} h_0 - h_0 = \int_0^t A e^{sA} h_0 \, ds.
\end{align*}
Defining the function
\begin{align*}
f : \mathbb{R}_+ \times H \rightarrow H, \quad f(s, y) := e^{sA} y,
\end{align*}
by Lemma~\ref{lemma-hg-rules} we have $f \in C_b^{1,2,{\rm loc}}(\mathbb{R}_+ \times H; H)$ with partial derivatives
\begin{align*}
D_s f(s, y) &= A e^{sA} y,
\\ D_y f(s, y) &= e^{sA},
\\ D_{yy} f(s, y) &= 0.
\end{align*}
By It\^{o}'s formula (Theorem~\ref{thm-Ito}) we get $\mathbb{P}$--almost surely
\begin{align*}
e^{tA} Y_t = f(t, Y_t) &= f(0, 0) + \int_0^t \big( D_s f(s, Y_s) + D_y f(s, Y_s) e^{-s A} \alpha(s, X_s) \big) ds 
\\ &\quad + \int_0^t D_y f(s, Y_s) e^{-s A} \sigma(s, X_s) dW_s
\\ &= \int_0^t \big( A e^{sA} Y_s + \alpha(s, X_s) \big) ds + \int_0^t \sigma(s, X_s) dW_s.
\end{align*}
Combining the previous identities, we obtain $\mathbb{P}$--almost surely
\begin{align*}
X_t &= e^{tA} (h_0 + Y_t) = h_0 + (e^{tA} h_0 - h_0) + e^{tA} Y_t
\\ &= h_0 + \int_0^t A e^{sA} h_0 \, ds + \int_0^t \big( A e^{sA} Y_s + \alpha(s, X_s) \big) ds + \int_0^t \sigma(s, X_s) dW_s
\\ &= h_0 + \int_0^t \big( A \underbrace{e^{sA} (h_0 + Y_s)}_{= X_s} + \alpha(s, X_s) \big) ds + \int_0^t \sigma(s, X_s) dW_s
\\ &= h_0 + \int_{0}^{t} \big( A X_s + \alpha(s, X_s) \big) ds + \int_{0}^{t} \sigma(s, X_s) dW_s, \quad t \geq 0,
\end{align*}
proving that $X$ is a strong solution to (\ref{SPDE}).
\end{proof}

\section{Stochastic convolution integrals}\label{sec-convolution}

In this section, we deal with the regularity of stochastic convolution integrals, which occur when dealing with mild solutions to SPDEs of the type (\ref{SPDE}). 

Let $E$ be a separable Banach space and let $(S_t)_{t \geq 0}$ be a $C_0$-semigroup on $E$. We start with the drift term.

\begin{lemma}\label{lemma-conv-stetig}
Let $f : \mathbb{R}_+ \rightarrow E$ be a measurable mapping such that
\begin{align*}
\int_0^t \| f(s) \| ds < \infty \quad \text{for all $t \geq 0$.}
\end{align*}
Then the mapping
\begin{align*}
F : \mathbb{R}_+ \rightarrow E, \quad F(t) := \int_0^t S_{t-s} f(s) ds
\end{align*}
is continuous.
\end{lemma}

\begin{proof}
Let $t \in \mathbb{R}_+$ be arbitrary. It suffices to prove that $F$ is right-continuous and left-continuous in $t$. 
\begin{enumerate}
\item Let $(t_n)_{n \in \mathbb{N}} \subset \mathbb{R}_+$ be a sequence such that $t_n \downarrow t$. Then for every $n \in \mathbb{N}$ we have
\begin{align*}
\| F(t) - F(t_n) \| &= \bigg\| \int_0^t S_{t-s} f(s) ds - \int_0^{t_n} S_{t_n - s} f(s) ds \bigg\|
\\ &= \bigg\| \int_0^t S_{t-s} f(s) ds - \int_0^{t} S_{t_n - s} f(s) ds - \int_t^{t_n} S_{t_n - s} f(s) ds \bigg\|
\\ &\leq \int_0^t \| S_{t-s} f(s) - S_{t_n - s} f(s) \| ds + \int_t^{t_n} \| S_{t_n -s} f(s) \| ds.
\end{align*}
By Lemma~\ref{lemma-C0-stetig} the mapping
\begin{align*}
\mathbb{R}_+ \times E \rightarrow E, \quad (u, x) \mapsto S_u x
\end{align*}
is continuous. Thus, taking into account estimate (\ref{Halbgruppe-Wachstum}) from Lemma~\ref{lemma-wachstum}, by Lebesgue's dominated convergence theorem we obtain
\begin{align*}
\| F(t) - F(t_n) \| \rightarrow 0 \quad \text{for $n \rightarrow \infty$.}
\end{align*}

\item Let $(t_n)_{n \in \mathbb{N}} \subset \mathbb{R}_+$ be a sequence such that $t_n \uparrow t$. Then for every $n \in \mathbb{N}$ we have
\begin{align*}
\| F(t) - F(t_n) \| &= \bigg\| \int_0^t S_{t-s} f(s) ds - \int_0^{t_n} S_{t_n - s} f(s) ds \bigg\|
\\ &= \bigg\| \int_0^{t_n} S_{t-s} f(s) ds - \int_{t_n}^{t} S_{t - s} f(s) ds - \int_0^{t_n} S_{t_n - s} f(s) ds \bigg\|
\\ &\leq \int_0^{t_n} \| S_{t-s} f(s) - S_{t_n - s} f(s) \| ds + \int_{t_n}^t \| S_{t -s} f(s) \| ds.
\end{align*}
Proceeding as in the previous situation, by Lebesgue's dominated convergence theorem we obtain
\begin{align*}
\| F(t) - F(t_n) \| \rightarrow 0 \quad \text{for $n \rightarrow \infty$.}
\end{align*}
\end{enumerate}
This completes the proof.
\end{proof}

\begin{proposition}\label{prop-conv-dt-stetig}
Let $X$ be a progressively measurable process satisfying
\begin{align*}
\mathbb{P} \bigg( \int_0^t \| X_s \| ds < \infty \bigg) = 1 \quad \text{for all $t \geq 0$.}
\end{align*}
Then the process $Y$ defined as
\begin{align*}
Y_t := \int_0^t S_{t-s} X_s ds, \quad t \geq 0
\end{align*}
is continuous and adapted.
\end{proposition}

\begin{proof}
The continuity of $Y$ is a consequence of Lemma~\ref{lemma-conv-stetig}. Moreover, $Y$ is adapted, because $X$ is progressively measurable.
\end{proof}

Now, we shall deal with stochastic convolution integrals driven by Wiener processes. Let $H$ be a separable Hilbert space and let $(S_t)_{t \geq 0}$ be a $C_0$-semigroup on $H$. Moreover, let $W$ be a trace class Wiener process on some separable Hilbert space~$\mathbb{H}$.

\begin{definition}
Let $X$ be a $L_2^0(H)$-valued predictable process such that
\begin{align*}
\mathbb{P} \bigg( \int_0^t \| X_s \|_{L_2^0(H)}^{2} ds < \infty \bigg) = 1 \quad \text{for all $t \geq 0$.}
\end{align*}
We define the \emph{stochastic convolution} $X \star W$ as
\begin{align*}
(X \star W)_t := \int_0^t S_{t-s} X_s dW_s, \quad t \geq 0.
\end{align*}
\end{definition}

We recall the following result concerning the regularity of stochastic convolutions.

\begin{proposition}\label{prop-conv-dW-stetig}
Let $X$ be a $L_2^0(H)$-valued predictable process such that one of the following two conditions is satisfied:
\begin{enumerate}
\item There exists a constant $p > 1$ such that
\begin{align*}
\mathbb{E} \bigg[ \int_0^t \| X_s \|_{L_2^0(H)}^{2p} ds \bigg] < \infty \quad \text{for all $t \geq 0$.}
\end{align*}
\item The semigroup $(S_t)_{t \geq 0}$ is a semigroup of pseudo-contractions, and we have
\begin{align*}
\mathbb{E} \bigg[ \int_0^t \| X_s \|_{L_2^0(H)}^{2} ds \bigg] < \infty \quad \text{for all $t \geq 0$.}
\end{align*}
\end{enumerate}
Then the stochastic convolution $X \star W$ has a continuous version. 
\end{proposition}

\begin{proof}
See \cite[Lemma~3.3]{Atma-book}.
\end{proof}

\section{Existence- and uniqueness results for SPDEs}\label{sec-existence}

In this section, we will present results concerning existence and uniqueness of solutions to the SPDE (\ref{SPDE}).

First, we recall the Banach fixed point theorem, which will be a basic result for proving the existence of mild solutions to (\ref{SPDE}).

\begin{definition}
Let $(E, d)$ be a metric space and let $\Phi : E \rightarrow E$ be a mapping.
\begin{enumerate}
\item The mapping $\Phi$ is called a \emph{contraction}, if for some constant $0 \leq L < 1$ we have
\begin{align*}
d(\Phi(x), \Phi(y)) \leq L \cdot d(x, y) \quad \text{for all $x, y \in E$.}
\end{align*}
\item An element $x \in E$ is called a \emph{fixed point} of $\Phi$, if we have
\begin{align*}
\Phi(x) = x.
\end{align*}
\end{enumerate}
\end{definition}

The following result is the well-known Banach fixed point theorem. Its proof can be found, e.g., in \cite[Theorem~3.48]{Aliprantis}. 

\begin{theorem}[Banach fixed point theorem]\label{thm-fixpunkt-Banach}
Let $E$ be a complete metric space and let $\Phi : E \rightarrow E$ be a contraction. Then the mapping $\Phi$ has a unique fixed point.
\end{theorem}

In this text, we shall use the following slight extension of the Banach fixed point theorem:

\begin{corollary}\label{cor-fixpunkt-n}
Let $E$ be a complete metric space and let $\Phi : E \rightarrow E$ be a mapping such that for some $n \in \mathbb{N}$ the mapping $\Phi^n$ is a contraction. Then the mapping $\Phi$ has a unique fixed point.
\end{corollary}

\begin{proof}
According to the Banach fixed point theorem (Theorem~\ref{thm-fixpunkt-Banach}) the mapping $\Phi^n$ has a unique fixed point, that is, there exists a unique element $x \in E$ such that $\Phi^n(x) = x$. Therefore, we have
\begin{align*}
\Phi(x) = \Phi(\Phi^n(x)) = \Phi^n(\Phi(x)),
\end{align*}
showing that $\Phi(x)$ is a fixed point of $\Phi^n$. Since $\Phi^n$ has a unique fixed point, we deduce that $\Phi(x) = x$, showing that $x$ is a fixed point of $\Phi$.

In order to prove uniqueness, let $y \in E$ be another fixed point of $\Phi$, that is, we have $\Phi(y) = y$. By induction, we obtain
\begin{align*}
\Phi^n(y) = \Phi^{n-1}(\Phi(y)) = \Phi^{n-1}(y) = \ldots = \Phi(y) = y,
\end{align*}
showing that $y$ is a fixed point of $\Phi^n$. Since the mapping $\Phi^n$ has exactly one fixed point, we obtain $x = y$.
\end{proof}

An indispensable tool for proving uniqueness of mild solutions to (\ref{SPDE}) will be the following version of Gronwall's inequality, see, e.g., \cite[Theorem~5.1]{EK}.

\begin{lemma}[Gronwall's inequality]\label{lemma-Gronwall}
Let $T \geq 0$ be fixed, let $f : [0, T] \rightarrow \mathbb{R}_+$ be a nonnegative continuous mapping, and let $\beta \geq 0$ be a constant such that
\begin{align*}
f(t) \leq \beta \int_0^t f(s) ds \quad \text{for all $t \in
[0, T]$.}
\end{align*}
Then we have $f \equiv 0$.
\end{lemma}

The following result shows that local Lipschitz continuity of $\alpha$ and $\sigma$ ensures the uniqueness of mild solutions to the SPDE (\ref{SPDE}). 

\begin{theorem}\label{thm-starke-Eind}
We suppose that for every $n \in \mathbb{N}$ there exists a constant $L_n \geq 0$ such that
\begin{align}\label{loc-Lipschitz-alpha}
\| \alpha(t, h_1) - \alpha(t, h_2) \| &\leq L_n \| h_1 - h_2 \|,
\\ \label{loc-Lipschitz-sigma} \| \sigma(t, h_1) - \sigma(t, h_2) \|_{L_2^0(H)} &\leq L_n \| h_1 - h_2 \|
\end{align}
for all $t \geq 0$ and all $h_1, h_2 \in H$ with $\| h_1 \|, \| h_2 \| \leq n$. Let $h_0, g_0 : \Omega \rightarrow H$ be two $\mathcal{F}_0$-measurable random variables, let $\tau > 0$ be a strictly positive stopping time, and let $X, Y$ be two local mild solutions to (\ref{SPDE}) with initial conditions $h_0$, $g_0$ and lifetime $\tau$. Then we have up to indistinguishability\footnote[1]{Two processes $X$ and $Y$ are called indistinguishable if the set $\{ \omega \in \Omega : X_t(\omega) \neq Y_t(\omega) \text{ for some } t \in \mathbb{R}_+ \}$ is a $\mathbb{P}$--nullset.}
\begin{align*}
X^{\tau} \mathbbm{1}_{\{h_0 = g_0\}} = Y^{\tau} \mathbbm{1}_{\{h_0 = g_0\}}.
\end{align*}
\end{theorem}

\begin{proof}
Defining the stopping times $(\tau_n)_{n \in \mathbb{N}}$ as
\begin{align*}
\tau_n := \tau \wedge \inf \{ t \geq 0 : \| X_t \| \geq n \} \wedge \inf \{ t \geq 0 : \| Y_t \| \geq n \},
\end{align*}
we have $\mathbb{P}(\tau_n \rightarrow \tau) = 1$. Let $n \in \mathbb{N}$ and $T \geq 0$ be arbitrary, and set 
\begin{align*}
\Gamma := \{ h_0 = g_0 \} \in \mathcal{F}_0. 
\end{align*}
The mapping
\begin{align*}
f : [0, T] \rightarrow \mathbb{R}, \quad f(t) := \mathbb{E} \big[ \mathbbm{1}_{\Gamma} \| X_{t \wedge \tau_n} - Y_{t \wedge \tau_n} \|^2 \big]
\end{align*}
is nonnegative, and it is continuous by Lebesgue's dominated convergence theorem. For all $t \in [0, T]$ we have
\begin{align*}
f(t) &= \mathbb{E} \big[ \mathbbm{1}_{\Gamma} \| X_{t \wedge \tau_n} - Y_{t \wedge \tau_n} \|^2 \big] \leq 3 \underbrace{\mathbb{E} \big[ \mathbbm{1}_{\Gamma} \| S_{t \wedge \tau_n} (h_0 - g_0) \|^2 \big]}_{=0} 
\\ &\quad + 3 \mathbb{E} \Bigg[ \mathbbm{1}_{\Gamma} \bigg\| \int_0^{t \wedge \tau_n} S_{(t \wedge \tau_n)-s} \big( \alpha(s, X_s) - \alpha(s, Y_s) \big) ds \bigg\|^2 \Bigg] 
\\ &\quad + 3 \mathbb{E} \Bigg[ \mathbbm{1}_{\Gamma} \bigg\| \int_0^{t \wedge \tau_n} S_{(t \wedge \tau_n)-s} \big( \sigma(s, X_s) - \sigma(s, Y_s) \big) dW_s \bigg\|^2 \Bigg]
\\ &= 3 \mathbb{E} \Bigg[ \bigg\| \int_0^{t \wedge \tau_n} \mathbbm{1}_{\Gamma} S_{(t \wedge \tau_n)-s} \big( \alpha(s, X_s) - \alpha(s, Y_s) \big) ds \bigg\|^2 \Bigg] 
\\ &\quad + 3 \mathbb{E} \Bigg[ \bigg\| \int_0^{t \wedge \tau_n} \mathbbm{1}_{\Gamma} S_{(t \wedge \tau_n)-s} \big( \sigma(s, X_s) - \sigma(s, Y_s) \big) dW_s \bigg\|^2 \Bigg],
\end{align*}
and hence, by the Cauchy-Schwarz inequality, the It\^{o} isometry (\ref{Ito-isom}), the growth estimate (\ref{Halbgruppe-Wachstum}) from Lemma~\ref{lemma-wachstum} and the local Lipschitz conditions (\ref{loc-Lipschitz-alpha}), (\ref{loc-Lipschitz-sigma}) we obtain
\begin{align*}
f(t) &\leq 3T \mathbb{E} \bigg[ \int_0^{t \wedge \tau_n} \big\| \mathbbm{1}_{\Gamma} S_{(t \wedge \tau_n)-s} \big( \alpha(s, X_s) - \alpha(s, Y_s) \big) \big\|^2 ds \bigg]
\\ &\quad + 3 \mathbb{E} \bigg[ \int_0^{t \wedge \tau_n} \big\| \mathbbm{1}_{\Gamma} S_{(t \wedge \tau_n)-s} \big( \sigma(s, X_s) - \sigma(s, Y_s) \big) \big\|_{L_2^0(H)}^2 ds \bigg]
\\ &\leq 3T \big( M e^{\omega T} \big)^2 \mathbb{E} \bigg[ \int_0^{t \wedge \tau_n} \mathbbm{1}_{\Gamma} \| \alpha(s, X_s) - \alpha(s, Y_s) \|^2 ds \bigg] 
\\ &\quad + 3 \big( M e^{\omega T} \big)^2 \mathbb{E} \bigg[ \int_0^{t \wedge \tau_n} \mathbbm{1}_{\Gamma} \| \sigma(s, X_s) - \sigma(s, Y_s) \|_{L_2^0(H)}^2 ds \bigg]
\\ &\leq 3(T+1) \big( M e^{\omega T} \big)^2 L_n^2 \int_0^t \mathbb{E} \big[ \mathbbm{1}_{\Gamma} \| X_{s \wedge \tau_n} - Y_{s \wedge \tau_n} \|^2 \big] ds 
\\ &= 3(T+1) \big( M e^{\omega T} \big)^2 L_n^2 \int_0^t f(s) ds.
\end{align*}
Using Gronwall's inequality (see Lemma~\ref{lemma-Gronwall}) we deduce that $f \equiv 0$. Thus, by the continuity of the sample paths of $X$ and $Y$, we obtain
\begin{align*}
\mathbb{P} \bigg( \bigcap_{t \geq 0} \{ X_{t \wedge \tau_n} \mathbbm{1}_{\Gamma} = Y_{t \wedge \tau_n} \mathbbm{1}_{\Gamma} \} \bigg) = 1 \quad \text{for all $n \in \mathbb{N}$,}
\end{align*}
and hence, by the continuity of the probability measure $\mathbb{P}$, we conclude that
\begin{align*}
\mathbb{P} \bigg( \bigcap_{t \geq 0} \{ X_{t \wedge \tau} \mathbbm{1}_{\Gamma} = Y_{t \wedge \tau} \mathbbm{1}_{\Gamma} \} \bigg) &= \mathbb{P} \bigg( \bigcap_{n \in \mathbb{N}} \bigcap_{t \geq 0} \{ X_{t \wedge \tau_n} \mathbbm{1}_{\Gamma} = Y_{t \wedge \tau_n} \mathbbm{1}_{\Gamma} \} \bigg)
\\ &= \lim_{n \rightarrow \infty} \mathbb{P} \bigg( \bigcap_{t \geq 0} \{ X_{t \wedge \tau_n} \mathbbm{1}_{\Gamma} = Y_{t \wedge \tau_n} \mathbbm{1}_{\Gamma} \} \bigg) = 1,
\end{align*}
which completes the proof.
\end{proof}

The local Lipschitz conditions (\ref{loc-Lipschitz-alpha}), (\ref{loc-Lipschitz-sigma}) are, in general, not sufficient in order to ensure the existence of mild solutions to the SPDE (\ref{SPDE}). Now, we will prove that the existence of mild solutions follows from global Lipschitz and linear growth conditions on $\alpha$ and $\sigma$. For this, we recall an auxiliary result which extends the It\^{o} isometry (\ref{Ito-isom}).

\begin{lemma}\label{lemma-est-Ito-2p}
Let $T \geq 0$ be arbitrary and let $X = (X_t)_{t \in [0, T]}$ be a $L_2^0(H)$-valued, predictable process such that
\begin{align*}
\mathbb{E} \bigg[ \int_0^T \| X_s \|_{L_2^0(H)}^2 ds \bigg] < \infty.
\end{align*}
Then, for every $p \geq 1$ we have
\begin{align*}
\mathbb{E} \Bigg[ \bigg\| \int_0^T X_s dW_s \bigg\|^{2p} \Bigg] \leq C_p \mathbb{E} \bigg[ \int_0^T \| X_s \|_{L_2^0(H)}^2 ds \bigg]^{p},
\end{align*}
where the constant $C_p > 0$ is given by
\begin{align*}
C_p = \big( p(2p-1) \big)^p \bigg( \frac{2p}{2p - 1} \bigg)^{2p^2}.
\end{align*}
\end{lemma}

\begin{proof}
See \cite[Lemma~3.1]{Atma-book}.
\end{proof}

\begin{theorem}\label{thm-existenz}
Suppose there exists a constant $L \geq 0$ such that 
\begin{align}\label{Lipschitz-alpha}
\| \alpha(t, h_1) - \alpha(t, h_2) \| &\leq L \| h_1 - h_2 \|,
\\ \label{Lipschitz-sigma} \| \sigma(t, h_1) - \sigma(t, h_2) \|_{L_2^0(H)} &\leq L \| h_1 - h_2 \|
\end{align}
for all $t \geq 0$ and all $h_1, h_2 \in H$, and suppose there exists a constant $K \geq 0$ such that
\begin{align}
\label{lin-growth-alpha} \| \alpha(t, h) \| &\leq K (1 + \| h \|),
\\ \label{lin-growth-sigma} \| \sigma(t, h) \|_{L_2^0(H)} &\leq K (1 + \| h \|)
\end{align}
for all $t \geq 0$ and all $h \in H$. Then, for every $\mathcal{F}_0$-measurable random variable $h_0 : \Omega \rightarrow H$ there exists a (up to indistinguishability) unique mild solution $X$ to (\ref{SPDE}).
\end{theorem}

\begin{proof}
The uniqueness of mild solutions to (\ref{SPDE}) is a direct consequence of Theorem~\ref{thm-starke-Eind}, and hence, we may concentrate on the existence proof, which we divide into several steps:

\noindent \emph{Step 1:} First, we suppose that the initial condition $h_0$ satisfies $\mathbb{E}[ \| h_0 \|^{2p} ] < \infty$ for some $p > 1$. Let $T \geq 0$ be arbitrary. We define the Banach space
\begin{align*}
L_T^{2p}(H) := L^{2p}(\Omega \times [0, T], \mathcal{P}_T, \mathbb{P} \otimes dt; H)
\end{align*}
and prove that the variation of constants equation
\begin{align}\label{var-const}
X_{t} = S_{t} h_0 + \int_{0}^{t} S_{t - s} \alpha(s, X_s) ds + \int_{0}^{t} S_{t - s} \sigma(s, X_s) dW_s, \quad t \in [0, T]
\end{align}
has a unique solution in the space $L_T^{2p}(H)$. This is done in the following three steps:

\noindent \emph{Step 1A:} For $X \in L_T^{2p}(H)$ we define the process $\Phi X$ by
\begin{align*}
(\Phi X)_t = S_t h_0 + \int_0^t S_{t-s} \alpha(s, X_s) ds + \int_0^t S_{t-s} \sigma(s, X_s) d W_s, \quad t \in [0, T].
\end{align*}
Then the process $\Phi X$ is well-defined. Indeed, by the growth estimate (\ref{Halbgruppe-Wachstum}), the linear growth condition (\ref{lin-growth-alpha}) and H\"{o}lder's inequality we have
\begin{align*}
&\mathbb{E} \bigg[ \int_0^T \| S_{t-s} \alpha(s, X_s) \| ds \bigg] \leq M e^{\omega T} \mathbb{E} \bigg[ \int_0^T \| \alpha(s, X_s) \| ds \bigg] 
\\ &\leq M e^{\omega T} K \mathbb{E} \bigg[ \int_0^T (1 + \|X_s\|) ds \bigg] 
= M e^{\omega T} K \bigg( T + \mathbb{E} \bigg[ \int_0^T \|X_s\| ds \bigg] \bigg)
\\ &\leq M e^{\omega T} K \bigg( T + T^{1 - \frac{1}{2p}} \mathbb{E} \bigg[ \int_0^T \|X_s\|^{2p} ds \bigg]^{1/2p} \bigg) < \infty.
\end{align*}
Furthermore, by the growth estimate (\ref{Halbgruppe-Wachstum}), the linear growth condition (\ref{lin-growth-sigma}) and H\"{o}lder's inequality we have
\begin{align*}
&\mathbb{E} \bigg[ \int_0^T \| S_{t-s} \sigma(s, X_s) \|_{L_2^0(H)}^2 ds \bigg] \leq \big( M e^{\omega T} \big)^2 \mathbb{E} \bigg[ \int_0^T \| \sigma(s, X_s) \|_{L_2^0(H)}^2 ds \bigg] 
\\ &\leq \big( M e^{\omega T} K \big)^2 \mathbb{E} \bigg[ \int_0^T (1 + \|X_s\|)^2 ds \bigg]
\\ &\leq 2 \big( M e^{\omega T} K \big)^2 \mathbb{E} \bigg[ \int_0^T ( 1 + \|X_s\|^2 ) ds \bigg]
= 2 \big( M e^{\omega T} K \big)^2 \bigg( T + \mathbb{E} \bigg[ \int_0^T \|X_s\|^2 ds \bigg] \bigg) 
\\ &\leq 2 \big( M e^{\omega T} K \big)^2 \bigg( T + T^{1 - \frac{1}{p}} \mathbb{E} \bigg[ \int_0^T \| X_s \|^{2p} \bigg]^{1/p} \bigg) < \infty.
\end{align*}
The previous two estimates show that $\Phi$ is a well-defined mapping on $L_T^{2p}(H)$.

\noindent \emph{Step 1B:} Next, we show that the mapping $\Phi$ maps $L_T^{2p}(H)$ into itself, that is, we have $\Phi : L_T^{2p}(H) \rightarrow L_T^{2p}(H)$. Indeed, let $X \in L_T^{2p}(H)$ be arbitrary. Defining the process $\Phi_{\alpha} X$ and $\Phi_{\sigma} X$ as
\begin{align*}
(\Phi_{\alpha} X)_t &:= \int_0^t S_{t-s} \alpha(s, X_s) ds, \quad t \in [0, T], 
\\ (\Phi_{\sigma} X)_t &:= \int_0^t S_{t-s} \sigma(s, X_s) d W_s, \quad t \in [0, T],
\end{align*}
we have
\begin{align*}
(\Phi X)_t = S_t h_0 + (\Phi_{\alpha} X)_t + (\Phi_{\sigma} X)_t, \quad t \in [0, T].
\end{align*}
By the growth estimate (\ref{Halbgruppe-Wachstum}) we have
\begin{align*}
\mathbb{E} \bigg[ \int_0^T \| S_t h_0 \|^{2p} dt \bigg] \leq \big( M e^{\omega T} \big)^{2p} T \mathbb{E} \big[ \| h_0 \|^{2p} \big] < \infty.
\end{align*}
By H\"{o}lder's inequality and the growth estimate (\ref{Halbgruppe-Wachstum}) we have
\begin{align*}
&\mathbb{E}\bigg[ \int_0^T \| (\Phi_{\alpha} X)_t \|^{2p} dt \bigg] = \mathbb{E} \bigg[ \int_0^T \bigg\| \int_0^t S_{t-s} \alpha(s, X_s) ds \bigg\|^{2p} dt \bigg] 
\\ &\leq t^{2p-1} \mathbb{E} \bigg[ \int_0^T \int_0^t \| S_{t-s} \alpha(s, X_s) \|^{2p} ds dt \bigg]
\\ &\leq T^{2p - 1} \big( M e^{\omega T} \big)^{2p} \mathbb{E} \bigg[ \int_0^T \int_0^t \| \alpha(s, X_s) \|^{2p} ds dt \bigg],
\end{align*}
and hence, by the linear growth condition (\ref{lin-growth-alpha}) and H\"{o}lder's inequality we obtain
\begin{align*}
&\mathbb{E}\bigg[ \int_0^T \| (\Phi_{\alpha} X)_t \|^{2p} dt \bigg] \leq T^{2p - 1} \big( M e^{\omega T} K \big)^{2p} \mathbb{E} \bigg[ \int_0^T \int_0^t ( 1 + \| X_s \| )^{2p} ds dt \bigg]
\\ &\leq T^{2p - 1} \big( M e^{\omega T} K \big)^{2p} 2^{2p - 1} \mathbb{E} \bigg[ \int_0^T \int_0^t ( 1 + \| X_s \|^{2p}) ds dt \bigg]
\\ &\leq (2T)^{2p - 1} \big( M e^{\omega T} K \big)^{2p} \bigg( \frac{T^2}{2} + T \mathbb{E} \bigg[ \int_0^T \| X_s \|^{2p} ds \bigg] \bigg) < \infty.
\end{align*}
Furthermore, by Lemma~\ref{lemma-est-Ito-2p} and the growth estimate (\ref{Halbgruppe-Wachstum}) we have
\begin{align*}
&\mathbb{E}\bigg[ \int_0^T \| (\Phi_{\sigma} X)_t \|^{2p} dt \bigg] = \mathbb{E} \bigg[ \int_0^T \bigg\| \int_0^t S_{t-s} \sigma(s, X_s) d W_s \bigg\|^{2p} dt \bigg] 
\\ &= \int_0^T \mathbb{E} \Bigg[ \bigg\| \int_0^t S_{t-s} \sigma(s, X_s) dW_s \bigg\|^{2p} \Bigg] dt 
\\ &\leq C_p \int_0^T \mathbb{E} \bigg[ \int_0^t \| S_{t-s} \sigma(s, X_s) \|_{L_2^0(H)}^2 ds \bigg]^p dt
\\ &\leq C_p \big( M e^{\omega T} \big)^{2p} \int_0^T \mathbb{E} \bigg[ \int_0^t \| \sigma(s, X_s) \|_{L_2^0(H)}^2 ds \bigg]^p dt
\end{align*}
and hence, by the linear growth condition (\ref{lin-growth-sigma}) and H\"{o}lder's inequality we obtain
\begin{align*}
&\mathbb{E}\bigg[ \int_0^T \| (\Phi_{\sigma} X)_t \|^{2p} dt \bigg] \leq C_p \big( M e^{\omega T} \big)^{2p} t^{p - 1} \mathbb{E} \bigg[ \int_0^T \int_0^t \| \sigma(s, X_s) \|_{L_2^0(H)}^{2p} ds dt \bigg] 
\\ &\leq C_p \big( M e^{\omega T} K \big)^{2p} T^{p - 1} \mathbb{E} \bigg[ \int_0^T \int_0^t ( 1 + \| X_s \| )^{2p} ds dt \bigg]
\\ &\leq C_p \big( M e^{\omega T} K \big)^{2p} T^{p - 1} 2^{2p - 1} \mathbb{E} \bigg[ \int_0^T \int_0^t ( 1 + \| X_s \|^{2p}) ds dt \bigg]
\\ &\leq C_p \big( M e^{\omega T} K \big)^{2p} 2^p (2T)^{p - 1} \bigg( \frac{T^2}{2} + T \mathbb{E} \bigg[ \int_0^T \| X_s \|^{2p} ds \bigg] \bigg) < \infty.
\end{align*}
The previous three estimates show that $\Phi X \in L^{2p}(H)$. Consequently, the mapping $\Phi$ maps $L_T^{2p}(H)$ into itself.

\noindent \emph{Step 1C:} Now, we show that for some index $n \in \mathbb{N}$ the mapping $\Phi^n$ is a contraction on $L_T^{2p}(H)$. Let $X, Y \in L_T^{2p}(H)$ and $t \in [0, T]$ be arbitrary. By H\"{o}lder's inequality, the growth estimate (\ref{Halbgruppe-Wachstum}) and the Lipschitz condition (\ref{Lipschitz-alpha}) we have
\begin{align*}
&\mathbb{E}\big[ \| (\Phi_{\alpha} X)_t - (\Phi_{\alpha} Y)_t \|^{2p} \big] 
= \mathbb{E} \Bigg[ \bigg\| \int_0^t S_{t-s} \alpha(s, X_s) ds - \int_0^t S_{t-s} \alpha(s, Y_s) ds \bigg\|^{2p} \Bigg] 
\\ &= \mathbb{E} \Bigg[ \bigg\| \int_0^t S_{t-s} \big( \alpha(s, X_s) - \alpha(s, Y_s) \big) ds \bigg\|^{2p} \Bigg]
\\ &\leq t^{2p-1} \mathbb{E} \bigg[ \int_0^t \big\| S_{t-s} \big( \alpha(s, X_s) - \alpha(s, Y_s) \big) \big\|^{2p} ds \bigg]
\\ &\leq T^{2p - 1} \big( M e^{\omega T} \big)^{2p} \mathbb{E} \bigg[ \int_0^t \| \alpha(s, X_s) - \alpha(s, Y_s) \|^{2p} ds \bigg] 
\\ &\leq T^{2p - 1} \big( M e^{\omega T} L \big)^{2p} \int_0^t \mathbb{E} \big[ \| X_s - Y_s \|^{2p} \big] ds.
\end{align*}
Furthermore, by Lemma~\ref{lemma-est-Ito-2p}, the growth estimate (\ref{Halbgruppe-Wachstum}), the Lipschitz condition (\ref{Lipschitz-sigma}) and H\"{o}lder's inequality we obtain
\begin{align*}
&\mathbb{E}\big[ \| (\Phi_{\sigma} X)_t - (\Phi_{\sigma} Y)_t \|^{2p} \big] = \mathbb{E} \Bigg[ \bigg\| \int_0^t S_{t-s} \sigma(s, X_s) d W_s - \int_0^t S_{t-s} \sigma(s, Y_s) d W_s \bigg\|^{2p} \Bigg] 
\\ &= \mathbb{E} \Bigg[ \bigg\| \int_0^t S_{t-s} \big( \sigma(s, X_s) - \sigma(s, Y_s) \big) dW_s \bigg\|^{2p} \Bigg]
\\ &\leq C_p \mathbb{E} \bigg[ \int_0^t \big\| S_{t-s} \big( \sigma(s, X_s) - \sigma(s, Y_s) \big) \big\|_{L_2^0(H)}^2 ds \bigg]^p
\\ &\leq C_p \big( M e^{\omega T} \big)^{2p} \mathbb{E} \Bigg[ \int_0^t \| \sigma(s, X_s) - \sigma(s, Y_s) \|_{L_2^0(H)}^2 ds \Bigg]^p
\\ &\leq C_p \big( M e^{\omega T} L \big)^{2p} \int_0^t \mathbb{E} \big[ \| X_s - Y_s \|^{2p} \big] ds.
\end{align*}
Therefore, defining the constant
\begin{align*}
C := 2^{2p-1} \big( T^{2p - 1} \big( M e^{\omega T} L \big)^{2p} + C_p \big( M e^{\omega T} L \big)^{2p} \big),
\end{align*}
by H\"{o}lder's inequality we get
\begin{align*}
&\mathbb{E}\big[ \| (\Phi X)_t - (\Phi Y)_t \|^{2p} \big]
\\ &\leq 2^{2p-1} \big( \mathbb{E}\big[ \| (\Phi_{\alpha} X)_t - (\Phi_{\alpha} Y)_t \|^{2p} \big] + \mathbb{E}\big[ \| (\Phi_{\sigma} X)_t - (\Phi_{\sigma} Y)_t \|^{2p} \big] \big)
\\ &\leq C \int_0^t \mathbb{E} \big[ \| X_s - Y_s \|^{2p} \big] ds.
\end{align*}
Thus, by induction for every $n \in \mathbb{N}$ we obtain
\begin{align*}
&\| \Phi^n X - \Phi^n Y \|_{L_T^{2p}(H)} = \bigg( \int_0^T \mathbb{E}\big[ \| (\Phi^n X)_{t_1} - (\Phi^n Y)_{t_1} \|^{2p} \big] dt_1 \bigg)^{1/2p}
\\ &\leq \bigg( C \int_0^T \bigg( \int_0^{t_1} \mathbb{E} \big[ \| (\Phi^{n-1} X)_{t_2} - (\Phi^{n-1} Y)_{t_2} \|^{2p} \big] dt_2 \bigg) dt_1 \bigg)^{1/2p}
\\ &\leq \ldots \leq \bigg( C^n \int_0^T \int_0^{t_1} \cdots \int_0^{t_{n-1}} \bigg( \int_0^{t_n} \mathbb{E} \big[ \|X_s - Y_s\|^{2p} \big] ds \bigg) dt_n \ldots dt_2 dt_1 \bigg)^{1/2p}
\\ &\leq \bigg( C^n \underbrace{\bigg( \int_0^T \int_0^{t_1} \cdots \int_0^{t_{n-1}} 1 dt_n \ldots dt_2 dt_1 \bigg)}_{= \frac{T^n}{n!}} \mathbb{E} \bigg[ \int_0^{T} \|X_s - Y_s\|^{2p} ds \bigg] \bigg)^{1/2p}
\\ &= \underbrace{\bigg( \frac{(C T)^n}{n!} \bigg)^{1/2p}}_{\rightarrow 0 \text{ for } n \rightarrow \infty} \| X - Y \|_{L_T^{2p}(H)}.
\end{align*}
Consequently, there exists an index $n \in \mathbb{N}$ such that $\Phi^n$ is a contraction, and hence, according to the extension of the Banach fixed point theorem (see Corollary~\ref{cor-fixpunkt-n}) the mapping $\Phi$ has a unique fixed point $X \in L_T^{2p}(H)$. This fixed point $X$ is a solution to the variation of constants equation (\ref{var-const}). Since $T \geq 0$ was arbitrary, there exists a process $X$ which is a solution of the variation of constants equation
\begin{align*}
X_{t} = S_{t} h_0 + \int_{0}^{t} S_{t - s} \alpha(s, X_s) ds + \int_{0}^{t} S_{t - s} \sigma(s, X_s) dW_s, \quad t \geq 0.
\end{align*}
\noindent \emph{Step 1D:} In order to prove that $X$ is a mild solution to (\ref{SPDE}), it remains to show that $X$ has a continuous version. By Lemma~\ref{lemma-C0-stetig}, the process 
\begin{align*}
t \mapsto S_t h_0, \quad t \geq 0
\end{align*}
is continuous, and, by Proposition~\ref{prop-conv-dt-stetig}, the process
\begin{align*}
\int_{0}^{t} S_{t - s} \alpha(s, X_s) ds, \quad t \geq 0
\end{align*}
is continuous, too. Moreover, for every $T \geq 0$ we have, by the linear growth condition (\ref{lin-growth-sigma}), H\"{o}lder's inequality, and since $X \in L_T^{2p}(H)$, the estimate
\begin{align*}
\mathbb{E} \bigg[ \int_0^T \| \sigma(s, X_s) \|_{L_2^0(H)}^{2p} ds \bigg] &\leq K^{2p} \mathbb{E} \bigg[ \int_0^T (1 + \| X_s \|)^{2p} ds \bigg]  
\\ &\leq K^{2p} 2^{2p-1} \mathbb{E} \bigg[ \int_0^T (1 + \| X_s \|^{2p}) ds \bigg]
\\ &= K (2K)^{2p-1} \bigg( T + \mathbb{E} \bigg[ \int_0^T \| X_s \|^{2p} ds \bigg] \bigg) < \infty.
\end{align*}
Thus, by Proposition~\ref{prop-conv-dW-stetig} the stochastic convolution $\sigma \star W$ given by
\begin{align*}
(\sigma \star W)_t = \int_{0}^{t} S_{t - s} \sigma(s, X_s) dW_s, \quad t \geq 0
\end{align*}
has a continuous version, and consequently, the process $X$ has a continuous version, too. This continuous version is a mild solution to (\ref{SPDE}).

\noindent \emph{Step 2:} Now let $h_0 : \Omega \rightarrow H$ be an arbitrary $\mathcal{F}_0$-measurable random variable. We define the sequence $(h_n)_{n \in \mathbb{N}}$ of $\mathcal{F}_0$-measurable random variables as
\begin{align*}
h_0^n := h_0 \mathbbm{1}_{\{ \| h_0 \| \leq n \}}, \quad n \in \mathbb{N}.
\end{align*}
Let $n \in \mathbb{N}$ be arbitrary. Then, as $h_0^n$ is bounded, we have $\mathbb{E}[ \| h_0^n \|^{2p} ] < \infty$ for all $p > 1$. By Step 1 the SPDE
\begin{align*}
\left\{
\begin{array}{rcl}
dX_t^n & = & (AX_t^n + \alpha(t, X_t^n))dt + \sigma(t, X_t^n) dW_t
\medskip
\\ X_0^n & = & h_0^n
\end{array}
\right.
\end{align*}
has a mild solution $X^n$. We define the sequence $(\Omega_n)_{n \in \mathbb{N}} \subset \mathcal{F}_0$ as 
\begin{align*}
\Omega_n := \{ \| h_0 \| \leq n \}, \quad n \in \mathbb{N}.
\end{align*}
Then we have $\Omega_n \subset \Omega_m$ for $n \leq m$, we have $\Omega = \bigcup_{n \in \mathbb{N}} \Omega_n$ and we have
\begin{align*}
\Omega_n \subset \{ h_0^n = h_0^m \} \subset \{ h_0^n = h_0 \} \quad \text{for all $n \leq m$.}
\end{align*}
Thus, by Theorem~\ref{thm-starke-Eind} we have (up to indistinguishability)
\begin{align*}
X^n \mathbbm{1}_{\Omega_n} = X^m \mathbbm{1}_{\Omega_n} \quad \text{for all $n \leq m$.}
\end{align*}
Consequently, the process
\begin{align*}
X := \lim_{n \rightarrow \infty} X^n \mathbbm{1}_{\Omega_n}
\end{align*}
is a well-defined, continuous and adapted process, and we have
\begin{align*}
X^n \mathbbm{1}_{\Omega_n} = X^m \mathbbm{1}_{\Omega_n} = X \mathbbm{1}_{\Omega_n} \quad \text{for all $n \leq m$.}
\end{align*}
Furthermore, we obtain $\mathbb{P}$--almost surely
\begin{align*}
X_t &= \lim_{n \rightarrow \infty} X_t^n \mathbbm{1}_{\Omega_n}
\\ &= \lim_{n \rightarrow \infty} \mathbbm{1}_{\Omega_n} \bigg( S_{t} h_0^n + \int_{0}^{t} S_{t - s} \alpha(s, X_s^n) ds + \int_{0}^{t} S_{t - s} \sigma(s, X_s^n) dW_s \bigg)
\\ &= \lim_{n \rightarrow \infty} \bigg( S_t (\mathbbm{1}_{\Omega_n} h_0^n) + \int_{0}^{t} \mathbbm{1}_{\Omega_n} S_{t - s} \alpha(s, X_s^n) ds + \int_{0}^{t} \mathbbm{1}_{\Omega_n} S_{t - s} \sigma(s, X_s^n) dW_s \bigg)
\\ &= \lim_{n \rightarrow \infty} \bigg( S_t (\mathbbm{1}_{\Omega_n} h_0) + \int_{0}^{t} \mathbbm{1}_{\Omega_n} S_{t - s} \alpha(s, X_s) ds + \int_{0}^{t} \mathbbm{1}_{\Omega_n} S_{t - s} \sigma(s, X_s) dW_s \bigg)
\\ &= \lim_{n \rightarrow \infty} \mathbbm{1}_{\Omega_n} \bigg( S_{t} h_0 + \int_{0}^{t} S_{t - s} \alpha(s, X_s) ds + \int_{0}^{t} S_{t - s} \sigma(s, X_s) dW_s \bigg)
\\ &= S_{t} h_0 + \int_{0}^{t} S_{t - s} \alpha(s, X_s) ds + \int_{0}^{t} S_{t - s} \sigma(s, X_s) dW_s, \quad t \geq 0,
\end{align*}
proving that $X$ is a mild solution to (\ref{SPDE}).
\end{proof}

\begin{remark}
For the proof of Theorem~\ref{thm-existenz} we have used Corollary~\ref{cor-fixpunkt-n}, which is a slight extension of the Banach fixed point theorem. Such an idea has been applied, e.g., in \cite{Barbara-pathwise}. 
\end{remark}

\begin{remark}
A recent method for proving existence and uniqueness of mild solutions to the SPDE (\ref{SPDE}) is the ``method of the moving frame'' presented in \cite{SPDE}, see also \cite{Tappe-Refine}. It allows to reduce SPDE problems to the study of SDEs in infinite dimension. In order to apply this method, we need that the semigroup $(S_t)_{t \geq 0}$ is a semigroup of pseudo-contractions.
\end{remark}

We close this section with a consequence about the existence of weak solutions.

\begin{corollary}\label{cor-existence}
Suppose that conditions (\ref{Lipschitz-alpha})--(\ref{lin-growth-sigma}) are fulfilled. Let $h_0 : \Omega \rightarrow H$ be a $\mathcal{F}_0$-measurable random variable such that $\mathbb{E}[ \| h_0 \|^{2p} ] < \infty$ for some $p > 1$. Then there exists a (up to indistinguishability) unique weak solution $X$ to (\ref{SPDE}).
\end{corollary}

\begin{proof}
According to Proposition~\ref{prop-schwach-mild}, every weak solution $X$ to (\ref{SPDE}) is also a mild solution to (\ref{SPDE}). Therefore, the uniqueness of weak solutions to (\ref{SPDE}) is a consequence of Theorem~\ref{thm-starke-Eind}.

It remains to prove the existence of a weak solution to (\ref{SPDE}).
Let $T \geq 0$ be arbitrary. By Theorem~\ref{thm-existenz} and its proof there exists a mild solution $X \in L_T^{2p}(H)$ to (\ref{SPDE}). By the linear growth condition (\ref{lin-growth-sigma}) and H\"{o}lder's inequality we obtain
\begin{align*}
&\mathbb{E} \bigg[ \int_0^T \| \sigma(s,X_s) \|_{L_2^0(H)}^2 \bigg] \leq K^2 \mathbb{E} \bigg[ \int_0^T (1 + \| X_s \|)^2 ds \bigg] \leq 2 K^2 \mathbb{E} \bigg[ \int_0^T (1 + \| X_s \|^2) ds \bigg]
\\ &= 2K^2 \bigg( T + \mathbb{E} \bigg[ \int_0^T \| X_s \|^2 ds \bigg] \bigg) \leq 2K^2 \bigg( T + T^{1 - \frac{1}{p}} \mathbb{E} \bigg[ \int_0^T \| X_s \|^{2p} ds \bigg]^{1/p} \bigg) < \infty,
\end{align*}
showing that condition (\ref{Fubini-condition}) is fulfilled.
Thus, by Proposition~\ref{prop-mild-schwach} the process $X$ is also a weak solution to (\ref{SPDE}). 
\end{proof}

\section{Invariant manifolds for weak solutions to SPDEs}\label{sec-manifolds}

In this section, we deal with invariant manifolds for time-homogeneous SPDEs of the type (\ref{SPDE}). This topic arises from the natural desire to express the solutions of the SPDE (\ref{SPDE}), which generally live in the infinite dimensional Hilbert space $H$, by means of a finite dimensional state process, and thus, to ensure larger analytical tractability. Our goal is to find conditions on the generator $A$ and the coefficients $\alpha$, $\sigma$ such that for every starting point of a finite dimensional submanifold the solution process stays on the submanifold.

We start with the required preliminaries about finite dimensional submanifolds in Hilbert spaces. In the sequel, let $H$ be a separable Hilbert space.
 
\begin{definition}
Let $m, k \in \mathbb{N}$ be positive integers. A subset $\mathcal{M} \subset H$ is called a $m$-dimensional \emph{$C^k$-submanifold} of $H$, if for every $h \in \mathcal{M}$ there exist an open neighborhood $U \subset H$ of $h$, an open set $V \subset \mathbb{R}^m$ and a mapping $\phi \in C^2(V;H)$ such that:
\begin{enumerate}
\item The mapping $\phi : V \rightarrow U \cap \mathcal{M}$ is a homeomorphism.

\item For all $y \in V$ the mapping $D\phi(y)$ is injective.
\end{enumerate}
The mapping $\phi$ is called a \emph{parametrization} of $\mathcal{M}$ around $h$.
\end{definition}

In what follows, let $\mathcal{M}$ be a $m$-dimensional $C^k$-submanifold of $H$.

\begin{lemma}\label{lemma-change-para}
Let $\phi_i : V_i \rightarrow U_i \cap \mathcal{M}$, $i=1,2$ be two parametrizations with $W := U_1 \cap U_2 \cap \mathcal{M} \neq \emptyset$. Then the mapping
\begin{align*}
\phi_1^{-1} \circ \phi_2 : \phi_2^{-1}(W) \rightarrow \phi_1^{-1}(W)
\end{align*}
is a $C^k$-diffeomorphism.
\end{lemma}

\begin{proof}
See \cite[Lemma~6.1.1]{fillnm}.
\end{proof}

\begin{corollary}\label{cor-tang-well}
Let $h \in \mathcal{M}$ be arbitrary and let $\phi_i : V_i \rightarrow U_i \cap \mathcal{M}$, $i=1, 2$ be two parametrizations of $\mathcal{M}$ around $h$. Then we have
\begin{align*}
D \phi_1(y_1)(\mathbb{R}^m) = D \phi_2(y_2)(\mathbb{R}^m),
\end{align*}
where $y_i = \phi_i^{-1}(h)$ for $i=1, 2$.
\end{corollary}

\begin{proof}
Since $U_1$ and $U_2$ are open neighborhoods of $h$, we have $W := U_1 \cap U_2 \cap \mathcal{M} \neq \emptyset$. Thus, by Lemma~\ref{lemma-change-para} the mapping
\begin{align*}
\phi_1^{-1} \circ \phi_2 : \phi_2^{-1}(W) \rightarrow \phi_1^{-1}(W)
\end{align*}
is a $C^k$-diffeomorphism. Using the chain rule, we obtain
\begin{align*}
D \phi_2(y_2)(\mathbb{R}^m) &= D (\phi_1 \circ (\phi_1^{-1} \circ \phi_2))(y_2)(\mathbb{R}^m) = D \phi_1(y_1) D(\phi_1^{-1} \circ \phi_2)(y_2)(\mathbb{R}^m) 
\\ &\subset D \phi_1(y_1)(\mathbb{R}^m),
\end{align*}
and, analogously, we prove that $D \phi_1(y_1)(\mathbb{R}^m) \subset D \phi_2(y_2)(\mathbb{R}^m)$.
\end{proof}

\begin{definition}\label{def-tang-raum}
Let $h \in \mathcal{M}$ be arbitrary. The \emph{tangent space} of $\mathcal{M}$ to $h$ is the subspace
\begin{align*}
T_h \mathcal{M} := D \phi(y)(\mathbb{R}^m),
\end{align*}
where $y = \phi^{-1}(h)$ and $\phi : V \rightarrow U \cap \mathcal{M}$ denotes a parametrization of $\mathcal{M}$ around $h$.
\end{definition}

\begin{remark}
Note that, according to Corollary~\ref{cor-tang-well}, the Definition~\ref{def-tang-raum} of the tangent space $T_h \mathcal{M}$ does not depend on the choice of the parametrization $\phi : V \rightarrow U \cap \mathcal{M}$.
\end{remark}

\begin{proposition}\label{prop-ext-para}
Let $h \in \mathcal{M}$ be arbitrary, and let $\phi : V \rightarrow U \cap \mathcal{M}$ be a parametrization of $\mathcal{M}$ around $h$. Then there exist an open set $V_0 \subset V$, an open neighborhood $U_0 \subset U$ of $h$, and a mapping $\hat{\phi} \in C_b^k(\mathbb{R}^m; H)$ with $\phi|_{V_0} = \hat{\phi}|_{V_0}$ such that $\phi|_{V_0} : V_0 \rightarrow U_0 \cap \mathcal{M}$ is a parametrization of $\mathcal{M}$ around $h$, too.
\end{proposition}

\begin{proof}
See \cite[Remark~6.1.1]{fillnm}.
\end{proof}

\begin{remark}\label{rem-ext-para}
By Proposition~\ref{prop-ext-para} we may assume that any parametrization $\phi : V \rightarrow U \cap \mathcal{M}$ has an extension $\phi \in C_b^k(\mathbb{R}^m; H)$.
\end{remark}

\begin{proposition}\label{prop-conv-para}
Let $D \subset H$ be a dense subset. For every $h_0 \in \mathcal{M}$ there exist $\zeta_1, \ldots, \zeta_m \in D$ and a parametrization $\phi : V \rightarrow U \cap \mathcal{M}$ around $h_0$ such that
\begin{align*}
\phi(\langle \zeta, h \rangle) = h \quad \text{for all $h \in U \cap \mathcal{M}$,}
\end{align*}
where we use the notation $\langle \zeta, h \rangle := (\langle \zeta_1, h \rangle, \ldots, \langle \zeta_m, h \rangle) \in \mathbb{R}^m$.
\end{proposition}

\begin{proof}
See \cite[Proposition~6.1.2]{fillnm}.
\end{proof}

\begin{proposition}\label{prop-proj-tang}
Let $\phi : V \rightarrow U \cap \mathcal{M}$ be a parametrization as in Proposition~\ref{prop-conv-para}. Then the following statements are true:
\begin{enumerate}
\item The elements $\zeta_1, \ldots, \zeta_m$ are linearly independent in $H$.

\item For every $h \in U \cap \mathcal{M}$ we have the direct sum decomposition
\begin{align}\label{decomp}
H = T_h \mathcal{M} \oplus \langle \zeta_1, \ldots, \zeta_m \rangle^{\perp}.
\end{align}

\item For every $h \in U \cap \mathcal{M}$ the mapping
\begin{align*}
\Pi_h = D\phi(y)(\langle \zeta, \bullet \rangle ) : H \rightarrow T_h \mathcal{M}, \quad \text{where $y = \langle \zeta, h \rangle$,}
\end{align*}
is the corresponding projection according to (\ref{decomp}) from $H$ onto $T_h \mathcal{M}$, that is, we have 
\begin{align*}
\Pi_h \in L(H), \quad \Pi_h^2 = \Pi_h, \quad {\rm ran}(\Pi_h) = T_h \mathcal{M} \quad \text{and} \quad {\rm ker}(\Pi_h) = \langle \zeta_1, \ldots, \zeta_m \rangle^{\perp}. 
\end{align*}
\end{enumerate}
\end{proposition}

\begin{proof}
See \cite[Lemma~6.1.3]{fillnm}.
\end{proof}

From now on, we assume that $\mathcal{M}$ is a $m$-dimensional $C^2$-submanifold of $H$.

\begin{proposition}\label{prop-zerlegung}
Let $\phi : V \rightarrow U \cap \mathcal{M}$ be a parametrization as in Proposition~\ref{prop-conv-para}. Furthermore, let $\sigma \in C^1(H)$ be a mapping such that
\begin{align}\label{sigma-pre-tangent}
\sigma(h) \in T_h \mathcal{M} \quad \text{for all $h \in U \cap \mathcal{M}$.}
\end{align}
Then, for every $h \in U \cap \mathcal{M}$ the direct sum decomposition of $D \sigma(h) \sigma(h)$ according to (\ref{decomp}) is given by
\begin{align}\label{decomp-element}
D \sigma(h) \sigma(h) = D\phi(y) ( \langle \zeta, D \sigma(h) \sigma(h) \rangle ) + D^2 \phi(y) ( \langle \zeta, \sigma(h) \rangle, \langle \zeta, \sigma(h) \rangle ),
\end{align}
where $y = \phi^{-1}(h)$.
\end{proposition}

\begin{proof}
Since $V$ is an open subset of $\mathbb{R}^m$, there exists $\epsilon > 0$ such that
\begin{align*}
y + t D\phi(y)^{-1} \sigma(h) \in V \quad \text{for all $t \in (-\epsilon, \epsilon)$.}
\end{align*}
Therefore, the curve
\begin{align*}
c : (-\epsilon, \epsilon) \rightarrow U \cap \mathcal{M}, \quad c(t) := \phi(y + t D\phi(y)^{-1} \sigma(h))
\end{align*}
is well-defined, and we have $c \in C^1((-\epsilon, \epsilon); H)$ with $c(0) = h$ and $c'(0) = \sigma(h)$. Hence, we have
\begin{align*}
\frac{d}{dt} \sigma(c(t))|_{t=0} = D\sigma(h)\sigma(h).
\end{align*}
Moreover, by condition (\ref{sigma-pre-tangent}) and Proposition~\ref{prop-proj-tang} we have
\begin{align*}
\frac{d}{dt} \sigma(c(t))|_{t=0} &= \frac{d}{dt} \Pi_{c(t)} \sigma(c(t))|_{t=0} = \frac{d}{dt} D \phi( \langle \zeta, c(t) \rangle ) ( \langle \zeta, \sigma(c(t)) \rangle )|_{t=0}
\\ &= D\phi(y) ( \langle \zeta, D \sigma(h) \sigma(h) \rangle ) + D^2 \phi(y) ( \langle \zeta, \sigma(h) \rangle, \langle \zeta, \sigma(h) \rangle ).
\end{align*}
The latter two identities prove the desired decomposition (\ref{decomp-element}).
\end{proof}

After these preliminaries, we shall study invariant manifolds for time-homogeneous SPDEs of the form
\begin{align}\label{SPDE-time}
\left\{
\begin{array}{rcl}
dX_t & = & (AX_t + \alpha(X_t))dt + \sigma(X_t) dW_t
\medskip
\\ X_0 & = & h_0
\end{array}
\right.
\end{align}
with measurable mappings $\alpha : H \rightarrow H$ and $\sigma : H \rightarrow L_2^0(H)$. As in the previous sections, the operator $A$ is the infinitesimal generator of a $C_0$-semigroup $(S_t)_{t \geq 0}$ on $H$. Note that, by (\ref{series-integral}), the SPDE (\ref{SPDE-time}) can be rewritten equivalently as
\begin{align}\label{SPDE-manifolds}
\left\{
\begin{array}{rcl}
dX_t & = & ( A X_t + \alpha(X_t) ) dt + \sum_{j \in \mathbb{N}} \sigma^j(X_t) d \beta_t^j \medskip
\\ X_0 & = & h_0,
\end{array}
\right.
\end{align}
where $(\beta^j)_{j \in \mathbb{N}}$ denotes the sequence of real-valued independent standard Wiener processes defined in (\ref{Wiener-series}), and where the mappings $\sigma^j : H \rightarrow H$, $j \in \mathbb{N}$ are given by $\sigma^j = \sqrt{\lambda_j} \sigma e_j$.

For the rest of this section, we assume that there exist a constant $L \geq 0$ such that
\begin{align}\label{alpha-Lipschitz-time}
\| \alpha(h_1) - \alpha(h_2) \| \leq L \| h_1 - h_2 \|, \quad h_1, h_2 \in H
\end{align}
and a sequence $(\kappa_j)_{j \in \mathbb{N}} \subset \mathbb{R}_+$ with $\sum_{j \in \mathbb{N}} \kappa_j^2 < \infty$ such that for every $j \in \mathbb{N}$ we have
\begin{align}\label{sigma-Lipschitz-time}
\| \sigma^j(h_1) - \sigma^j(h_2) \| &\leq \kappa_j \| h_1 - h_2 \|, \quad h_1, h_2 \in H
\\ \label{sigma-lin-growth-time} \| \sigma^j(h) \| &\leq \kappa_j ( 1 + \| h \| ), \quad h \in H.
\end{align}

\begin{proposition}\label{prop-Loesung-konst}
For every $h_0 \in H$ there exists a (up to indistinguishability) unique weak solution to (\ref{SPDE-manifolds}).
\end{proposition}

\begin{proof}
By (\ref{sigma-Lipschitz-time}), for all $h_1, h_2 \in H$ we have
\begin{align*}
\| \sigma(h_1) - \sigma(h_2) \|_{L_2^0(H)} = \bigg( \sum_{j \in \mathbb{N}} \| \sigma^j(h_1) - \sigma^j(h_2) \|^2 \bigg)^{1/2}
\leq \bigg( \sum_{j \in \mathbb{N}} \kappa_j^2 \bigg)^{1/2} \| h_1 - h_2 \|.
\end{align*}
Moreover, by (\ref{alpha-Lipschitz-time}), for every $h \in H$ we have
\begin{align*}
\| \alpha(h) \| \leq \| \alpha(h) - \alpha(0) \| + \| \alpha(0) \| \leq L \| h \| + \| \alpha(0) \| \leq \max \{ L, \| \alpha(0) \| \} (1 + \| h \|),
\end{align*}
and, by (\ref{sigma-lin-growth-time}) we obtain
\begin{align*}
\| \sigma(h) \|_{L_2^0(H)} = \bigg( \sum_{j \in \mathbb{N}} \| \sigma^j(h) \|^2 \bigg)^{1/2} \leq \bigg( \sum_{j \in \mathbb{N}} \kappa_j^2 \bigg)^{1/2} (1 + \| h \|).
\end{align*}
Therefore, conditions (\ref{Lipschitz-alpha})--(\ref{lin-growth-sigma}) are fulfilled, and hence, applying Corollary~\ref{cor-existence} completes the proof.
\end{proof}

Recall that $\mathcal{M}$ denotes a finite dimensional $C^2$-submanifold of $H$.

\begin{definition}
The submanifold $\mathcal{M}$ is called \emph{locally invariant} for (\ref{SPDE-manifolds}), if for every $h_0 \in \mathcal{M}$ there exists a local weak solution $X$ to (\ref{SPDE-manifolds}) with some lifetime $\tau > 0$ such that
\begin{align*}
X_{t \wedge \tau} \in \mathcal{M} \quad \text{for all $t \geq 0$,} \quad \text{$\mathbb{P}$--almost surely.}
\end{align*}
\end{definition}

In order to investigate local invariance of $\mathcal{M}$, we will assume, from now on, that $\sigma^j \in C^1(H)$ for all $j \in \mathbb{N}$.

\begin{lemma}\label{lemma-Damir}
The following statements are true:
\begin{enumerate}
\item For every $h \in H$ we have 
\begin{align}\label{sum-Damir-finite}
\sum_{j \in \mathbb{N}} \| D \sigma^j(h) \sigma^j(h) \| < \infty.
\end{align}

\item The mapping
\begin{align}\label{mapping-Damir}
H \rightarrow H, \quad h \mapsto \sum_{j \in \mathbb{N}} D \sigma^j(h) \sigma^j(h)
\end{align}
is continuous.
\end{enumerate}
\end{lemma}

\begin{proof}
By (\ref{sigma-Lipschitz-time}) and (\ref{sigma-lin-growth-time}), for every $h \in H$ we have
\begin{align*}
\sum_{j \in \mathbb{N}} \| D \sigma^j(h) \sigma^j(h) \| \leq \sum_{j \in \mathbb{N}} \| D \sigma^j(h) \| \, \| \sigma^j(h) \| \leq (1 + \| h \|) \sum_{j \in \mathbb{N}} \kappa_j^2 < \infty,
\end{align*}
showing (\ref{sum-Damir-finite}). Moreover, for every $j \in \mathbb{N}$ the mapping
\begin{align*}
H \mapsto H, \quad D \sigma^j(h) \sigma^j(h)
\end{align*}
is continuous, because for all $h_1, h_2 \in H$ we have
\begin{align*}
&\| D \sigma^j(h_1) \sigma^j(h_1) - D \sigma^j(h_2) \sigma^j(h_2) \|
\\ &\leq \| D \sigma^j(h_1) \sigma^j(h_1) - D \sigma^j(h_1) \sigma^j(h_2) \| + \| D \sigma^j(h_1) \sigma^j(h_2) - D \sigma^j(h_2) \sigma^j(h_2) \|
\\ &\leq \| D \sigma^j(h_1) \| \, \| \sigma^j(h_1) - \sigma^j(h_2) \| + \| D \sigma^j(h_1) - D \sigma^j(h_2) \| \, \| \sigma^j(h_2) \|.
\end{align*}
Let $\nu$ be the counting measure on $(\mathbb{N}, \mathfrak{P}(\mathbb{N}))$, which is given by $\nu(\{j\}) = 1$ for all $j \in \mathbb{N}$. Then we have
\begin{align*}
\sum_{j \in \mathbb{N}} D \sigma^j(h) \sigma^j(h) = \int_{\mathbb{N}} D \sigma^j(h) \sigma^j(h) \nu(dj).
\end{align*}
Hence, because of the estimate
\begin{align*}
\| D \sigma^j(h) \sigma^j(h) \| \leq (1+ \| h \|) \kappa_j^2, \quad h \in H \text{ and } j \in \mathbb{N}
\end{align*}
the continuity of the mapping (\ref{mapping-Damir}) is a consequence of Lebesgue's dominated convergence theorem.
\end{proof}

For a mapping $\phi \in C_b^2(\mathbb{R}^m; H)$ and elements $\zeta_1, \ldots, \zeta_m \in \mathcal{D}(A^*)$ we define the mappings $\alpha_{\phi, \zeta} : \mathbb{R}^m \rightarrow \mathbb{R}^m$ and $\sigma_{\phi, \zeta}^j : \mathbb{R}^m \rightarrow \mathbb{R}^m$, $j \in \mathbb{N}$ as
\begin{align*}
\alpha_{\phi, \zeta}(y) &:= \langle A^* \zeta, \phi(y) \rangle + \langle \zeta, \alpha(\phi(y)) \rangle,
\\ \sigma_{\phi, \zeta}^j(y) &:= \langle \zeta, \sigma^j(\phi(y)) \rangle.
\end{align*}

\begin{proposition}
Let $\phi \in C_b^2(\mathbb{R}^m; H)$ and $\zeta_1, \ldots, \zeta_m \in \mathcal{D}(A^*)$ be arbitrary. Then, for every $y_0 \in \mathbb{R}^m$ there exists a (up to indistinguishability) unique strong solution to the SDE
\begin{align}\label{SDE-Y}
\left\{
\begin{array}{rcl}
dY_t & = & \alpha_{\phi, \zeta}(Y_t) dt + \sum_{j \in \mathbb{N}} \sigma_{\phi, \zeta}^j(Y_t) d\beta_t^j \medskip
\\ Y_0 & = & y_0.
\end{array}
\right.
\end{align}
\end{proposition}

\begin{proof}
By virtue of the assumption $\phi \in C_b^2(\mathbb{R}^m; H)$ and (\ref{alpha-Lipschitz-time})--(\ref{sigma-lin-growth-time}), there exist a constant $\tilde{L} \geq 0$ such that
\begin{align*}
\| \alpha_{\phi, \zeta}(y_1) - \alpha_{\phi, \zeta}(y_2) \|_{\mathbb{R}^m} \leq \tilde{L} \| y_1 - y_2 \|_{\mathbb{R}^m}, \quad y_1, y_2 \in \mathbb{R}^m
\end{align*}
and a sequence $(\tilde{\kappa}_j)_{j \in \mathbb{N}} \subset \mathbb{R}_+$ with $\sum_{j \in \mathbb{N}} \tilde{\kappa}_j^2 < \infty$ such that for every $j \in \mathbb{N}$ we have
\begin{align*}
\| \sigma_{\phi, \zeta}^j(y_1) - \sigma_{\phi, \zeta}^j(y_2) \|_{\mathbb{R}^m} &\leq \tilde{\kappa}_j \| y_1 - y_2 \|_{\mathbb{R}^m}, \quad y_1, y_2 \in \mathbb{R}^m
\\ \| \sigma_{\phi, \zeta}^j(y) \|_{\mathbb{R}^m} &\leq \tilde{\kappa}_j ( 1 + \| y \|_{\mathbb{R}^m} ), \quad y \in \mathbb{R}^m.
\end{align*}
Therefore, by Proposition~\ref{prop-Loesung-konst}, for every $y_0 \in \mathbb{R}^m$ there exists a (up to indistinguishability) unique weak solution to (\ref{SDE-Y}), which, according to Proposition~\ref{prop-SPDE-normstetig} is also a strong solution to (\ref{SDE-Y}). The uniqueness of strong solutions to (\ref{SDE-Y}) is a consequence of Proposition~\ref{prop-SPDE-normstetig} and Theorem~\ref{thm-starke-Eind}.
\end{proof}

Now, we are ready to formulate and prove our main result of this section.

\begin{theorem}\label{thm-manifolds}
The following statements are equivalent:
\begin{enumerate}
\item The submanifold $\mathcal{M}$ is locally invariant for (\ref{SPDE-manifolds}).

\item We have
\begin{align}\label{domain}
&\mathcal{M} \subset \mathcal{D}(A),
\\ \label{sigma-tang} &\sigma^j(h) \in T_h \mathcal{M} \quad \text{for all $h \in \mathcal{M}$ and all $j \in \mathbb{N}$,}
\\ \label{alpha-tang} &Ah + \alpha(h) - \frac{1}{2} \sum_{j \in \mathbb{N}} D \sigma^j(h) \sigma^j(h) \in T_h \mathcal{M} \quad \text{for all $h \in \mathcal{M}$.}
\end{align}

\item The operator $A$ is continuous on $\mathcal{M}$, and for each $h_0 \in \mathcal{M}$ there exists a local strong solution $X$ to (\ref{SPDE-manifolds}) with some lifetime $\tau > 0$ such that
\begin{align*}
X_{t \wedge \tau} \in \mathcal{M} \quad \text{for all $t \geq 0$,} \quad \text{$\mathbb{P}$--almost surely.}
\end{align*}
\end{enumerate}
\end{theorem}

\begin{proof}
(1) $\Rightarrow$ (2): Let $h \in \mathcal{M}$ be arbitrary. By Proposition~\ref{prop-conv-para} and Remark~\ref{rem-ext-para} there exist elements $\zeta_1, \ldots, \zeta_m \in \mathcal{D}(A^*)$ and a parametrization $\phi : V \rightarrow U \cap \mathcal{M}$ around $h$ such that the inverse $\phi^{-1} : U \cap \mathcal{M} \rightarrow V$ is given by $\phi^{-1} = \langle \zeta, \bullet \rangle$, and $\phi$ has an extension $\phi \in C_b^2(\mathbb{R}^m; H)$. Since the submanifold $\mathcal{M}$ is locally invariant for (\ref{SPDE-manifolds}), there exists a local weak solution $X$ to (\ref{SPDE-manifolds}) with initial condition $h$ and some lifetime $\varrho > 0$ such that
\begin{align*}
X_{t \wedge \varrho} \in \mathcal{M} \quad \text{for all $t \geq 0$,} \quad \text{$\mathbb{P}$--almost surely.}
\end{align*}
Since $U$ is an open neighborhood of $h$, there exists $\epsilon > 0$ such that $\overline{B_{\epsilon}(h)} \subset U$, where $B_{\epsilon}(h)$ denotes the open ball
\begin{align*}
B_{\epsilon}(h) = \{ g \in H : \| g - h \| < \epsilon \}. 
\end{align*}
We define the stopping time
\begin{align*}
\tau := \varrho \wedge \inf \{ t \geq 0 : X_t \notin B_{\epsilon}(h) \}.
\end{align*}
Since the process $X$ has continuous sample paths and satisfies $X_0 = h$, we have $\tau > 0$ and $\mathbb{P}$--almost surely
\begin{align*}
X_{t \wedge \tau} \in U \cap \mathcal{M} \quad \text{for all $t \geq 0$.}
\end{align*}
Defining the $\mathbb{R}^m$-valued process $Y := \langle \zeta, X \rangle$ we have $\mathbb{P}$--almost surely
\begin{align*}
Y_{t \wedge \tau} \in V \quad \text{for all $t \geq 0$.}
\end{align*}
Moreover, since $X$ is a weak solution to (\ref{SPDE-manifolds}) with initial condition $h$, setting $y := \langle \zeta, h \rangle \in V$ we have $\mathbb{P}$--almost surely
\begin{align*}
Y_{t \wedge \tau} &= \langle \zeta, h \rangle + \int_0^{t \wedge \tau} \big( \langle A^* \zeta, X_s \rangle + \langle \zeta, \alpha(X_s) \rangle \big) ds + \sum_{j \in \mathbb{N}} \int_0^{t \wedge \tau} \langle \zeta, \sigma^j(X_s) \rangle d\beta_s^j
\\ &= \langle \zeta, h \rangle + \int_0^{t \wedge \tau} \alpha_{\phi, \zeta}(\langle \zeta, X_s \rangle) ds + \sum_{j \in \mathbb{N}} \int_0^{t \wedge \tau} \sigma_{\phi, \zeta}^j(\langle \zeta, X_s \rangle) d\beta_s^j
\\ &= y + \int_0^{t \wedge \tau} \alpha_{\phi, \zeta}(Y_s) ds + \sum_{j \in \mathbb{N}} \int_0^{t \wedge \tau} \sigma_{\phi, \zeta}^j(Y_s) d\beta_s^j, \quad t \geq 0,
\end{align*}
showing that $Y$ is a local strong solution to (\ref{SDE-Y}) with initial condition $y$. By It\^{o}'s formula (Theorem~\ref{thm-Ito}) we obtain $\mathbb{P}$--almost surely
\begin{align*}
X_{t \wedge \tau} &= \phi(Y_{t \wedge \tau}) 
\\ &= h + \int_0^{t \wedge \tau} \bigg( D \phi(Y_s) \alpha_{\phi, \zeta}(Y_s) + \frac{1}{2} \sum_{j \in \mathbb{N}} D^2 \phi(Y_s)(\sigma_{\phi, \zeta}^j(Y_s), \sigma_{\phi, \zeta}^j(Y_s)) \bigg) ds
\\ &\quad + \sum_{j \in \mathbb{N}} \int_0^{t \wedge \tau} D \phi(Y_s) \sigma_{\phi, \zeta}^j(Y_s) d\beta_s^j, \quad t \geq 0.
\end{align*}
Now, let $\xi \in \mathcal{D}(A^*)$ be arbitrary. Then we have $\mathbb{P}$--almost surely
\begin{equation}\label{manifold-eq-1}
\begin{aligned}
\langle \xi, X_{t \wedge \tau} \rangle &= \langle \xi, h \rangle 
\\ &\quad + \int_0^{t \wedge \tau} \Big\langle \xi, D \phi(Y_s) \alpha_{\phi, \zeta}(Y_s) + \frac{1}{2} \sum_{j \in \mathbb{N}} D^2 \phi(Y_s)(\sigma_{\phi, \zeta}^j(Y_s), \sigma_{\phi, \zeta}^j(Y_s)) \Big\rangle ds
\\ &\quad + \sum_{j \in \mathbb{N}} \int_0^{t \wedge \tau} \langle \xi, D \phi(Y_s) \sigma_{\phi, \zeta}^j(Y_s) \rangle d\beta_s^j, \quad t \geq 0.
\end{aligned}
\end{equation}
On the other hand, since $X$ is a local weak solution to (\ref{SPDE-manifolds}) with initial condition $h$ and lifetime $\tau$, we have $\mathbb{P}$--almost surely for all $t \geq 0$ the identity
\begin{equation}\label{manifold-eq-2}
\begin{aligned}
\langle \xi, X_{t \wedge \tau} \rangle &= \langle \xi, h \rangle + \int_0^{t \wedge \tau} \big( \langle A^* \xi, X_s \rangle + \langle \xi, \alpha(X_s) \rangle \big) ds + \sum_{j \in \mathbb{N}} \int_0^{t \wedge \tau} \langle \xi, \sigma^j(X_s) \rangle d\beta_s^j.
\end{aligned}
\end{equation}
Combining (\ref{manifold-eq-1}) and (\ref{manifold-eq-2}) yields up to indistinguishability
\begin{align}\label{can-decomp}
B + M = 0,
\end{align}
where the processes $B$ and $M$ are defined as
\begin{align*}
B_t &:= \int_0^{t \wedge \tau} \bigg( \langle A^* \xi, X_s \rangle + \Big\langle \xi, \alpha(X_s) - D \phi(Y_s) \alpha_{\phi, \zeta}(Y_s) 
\\ &\quad\quad\quad\quad\quad - \frac{1}{2} \sum_{j \in \mathbb{N}} D^2 \phi(Y_s)(\sigma_{\phi, \zeta}^j(Y_s), \sigma_{\phi, \zeta}^j(Y_s)) \Big\rangle \bigg) ds, \quad t \geq 0,
\\ M_t &:= \sum_{j \in \mathbb{N}} \int_0^{t \wedge \tau} \langle \xi, \sigma^j(X_s) - D \phi(Y_s) \sigma_{\phi, \zeta}^j(Y_s) \rangle d\beta_s^j, \quad t \geq 0.
\end{align*}
The process $B + M$ is a continuous semimartingale with canonical decomposition (\ref{can-decomp}). Since the canonical decomposition of a continuous semimartingale is unique up to indistinguishability, we deduce that $B = M = 0$ up to indistinguishability. Using the It\^{o} isometry (\ref{Ito-isom}) we obtain $\mathbb{P}$--almost surely
\begin{align*}
&\int_0^{t \wedge \tau} \bigg( \langle A^* \xi, X_s \rangle + \Big\langle \xi, \alpha(X_s) - D \phi(Y_s) \alpha_{\phi, \zeta}(Y_s) 
\\ &\quad\quad\quad\,\,\, - \frac{1}{2} \sum_{j \in \mathbb{N}} D^2 \phi(Y_s)(\sigma_{\phi, \zeta}^j(Y_s), \sigma_{\phi, \zeta}^j(Y_s)) \Big\rangle \bigg) ds = 0, \quad t \geq 0,
\\ &\int_0^{t \wedge \tau} \sum_{j \in \mathbb{N}} |\langle \xi, \sigma^j(X_s) - D \phi(Y_s) \sigma_{\phi, \zeta}^j(Y_s) \rangle|^2 ds = 0, \quad t \geq 0.
\end{align*}
By the continuity of the processes $X$ and $Y$ we obtain for all $\xi \in \mathcal{D}(A^*)$ the identities
\begin{align*}
&\langle A^* \xi, h \rangle + \Big\langle \xi, \alpha(h) - D \phi(y) \alpha_{\phi, \zeta}(y) - \frac{1}{2} \sum_{j \in \mathbb{N}} D^2 \phi(y)(\sigma_{\phi, \zeta}^j(y), \sigma_{\phi, \zeta}^j(y)) \Big\rangle = 0,
\\ &\langle \xi, \sigma^j(h) - D \phi(y) \sigma_{\phi, \zeta}^j(y) \rangle = 0, \quad j \in \mathbb{N}.
\end{align*}
Consequently, the mapping $\xi \mapsto \langle A^* \xi, h \rangle$ is continuous on $\mathcal{D}(A^*)$, and hence we have $h \in \mathcal{D}(A^{**})$ by the definition (\ref{def-domain-adjoint}). By Proposition~\ref{prop-A-adj-dicht} we have $A = A^{**}$, and thus we obtain $h \in \mathcal{D}(A)$, proving (\ref{domain}). By Proposition~\ref{prop-A-adj-dicht}, the domain $\mathcal{D}(A^*)$ is dense in $H$, and thus
\begin{align*}
\sigma^j(h) = D \phi(y) \sigma_{\phi, \zeta}^j(y) \in T_h \mathcal{M}, \quad j \in \mathbb{N},
\end{align*}
showing (\ref{sigma-tang}). Moreover, for all $\xi \in \mathcal{D}(A^*)$ we have
\begin{align*}
&\Big\langle \xi, Ah + \alpha(h) - D \phi(y) \alpha_{\phi, \zeta}(y) - \frac{1}{2} \sum_{j \in \mathbb{N}} D^2 \phi(y)(\sigma_{\phi, \zeta}^j(y), \sigma_{\phi, \zeta}^j(y)) \Big\rangle = 0.
\end{align*}
Since the domain $\mathcal{D}(A^*)$ is dense in $H$, together with Proposition~\ref{prop-zerlegung} we obtain
\begin{align*}
&Ah + \alpha(h) - \frac{1}{2} \sum_{j \in \mathbb{N}} D\sigma^j(h)\sigma^j(h) 
\\ &= Ah + \alpha(h) - \frac{1}{2} \sum_{j \in \mathbb{N}} \big( D\phi(y) ( \langle \zeta, D \sigma^j(h) \sigma^j(h) \rangle ) + D^2 \phi(y) (\sigma_{\phi, \zeta}^j(y), \sigma_{\phi, \zeta}^j(y)) \big)
\\ &= D \phi(y) \alpha_{\phi, \zeta}(y) - \frac{1}{2} \sum_{j \in \mathbb{N}} D\phi(y) ( \langle \zeta, D \sigma^j(h) \sigma^j(h) \rangle ) 
\\ &= D\phi(y) \bigg( \alpha_{\phi, \zeta}(y) - \frac{1}{2} \sum_{j \in \mathbb{N}} \langle \zeta, D \sigma^j(h) \sigma^j(h) \rangle \bigg) \in T_h \mathcal{M},
\end{align*}
which proves (\ref{alpha-tang}).

\noindent(2) $\Rightarrow$ (1): Let $h_0 \in \mathcal{M}$ be arbitrary. By Proposition~\ref{prop-conv-para} and Remark~\ref{rem-ext-para} there exist $\zeta_1, \ldots, \zeta_m \in \mathcal{D}(A^*)$ and a parametrization $\phi : V \rightarrow U \cap \mathcal{M}$ around $h_0$ such that the inverse $\phi^{-1} : U \cap \mathcal{M} \rightarrow V$ is given by $\phi^{-1} = \langle \zeta, \bullet \rangle$, and $\phi$ has an extension $\phi \in C_b^2(\mathbb{R}^m; H)$. Let $h \in U \cap \mathcal{M}$ be arbitrary and set $y := \langle \zeta, h \rangle \in V$. By relations (\ref{domain}), (\ref{alpha-tang}) and Proposition~\ref{prop-proj-tang} we obtain
\begin{align*}
&Ah + \alpha(h) - \frac{1}{2} \sum_{j \in \mathbb{N}} D\sigma^j(h)\sigma^j(h) = D\phi(y) \bigg( \Big\langle \zeta, Ah + \alpha(h) - \frac{1}{2} \sum_{j \in \mathbb{N}} D\sigma^j(h)\sigma^j(h) \Big\rangle \bigg),
\end{align*}
and thus
\begin{align*}
Ah &= D\phi(y) \bigg( \langle A^* \zeta, h \rangle + \Big\langle \zeta, \alpha(h) - \frac{1}{2} \sum_{j \in \mathbb{N}} D\sigma^j(h)\sigma^j(h) \Big\rangle \bigg) 
\\ &\quad - \alpha(h) + \frac{1}{2} \sum_{j \in \mathbb{N}} D\sigma^j(h)\sigma^j(h).
\end{align*}
Together with Lemma~\ref{lemma-Damir}, this proves the continuity of $A$ on $U \cap \mathcal{M}$. Since $h_0 \in \mathcal{M}$ was arbitrary, this proves that $A$ is continuous on $\mathcal{M}$. 

Furthermore, by (\ref{sigma-tang}) and Proposition~\ref{prop-proj-tang} we have
\begin{align}\label{tangent-1}
\sigma^j(h) = D\phi(y) \sigma_{\phi, \zeta}^j(h) \quad \text{for every $j \in \mathbb{N}$.}
\end{align}
Moreover, by (\ref{domain}), (\ref{alpha-tang}) and Propositions~\ref{prop-proj-tang} and \ref{prop-zerlegung} we obtain
\begin{align*}
&Ah + \alpha(h) - \frac{1}{2} \sum_{j \in \mathbb{N}} D\sigma^j(h)\sigma^j(h) = D\phi(y) \bigg( \Big\langle \zeta, Ah + \alpha(h) - \frac{1}{2} \sum_{j \in \mathbb{N}} D\sigma^j(h)\sigma^j(h) \Big\rangle \bigg)
\\ &= D\phi(y) \big( \langle A^* \zeta, h \rangle + \langle \zeta, \alpha(h) \rangle \big) - \frac{1}{2} \sum_{j \in \mathbb{N}} D\phi(y) \langle \zeta, D\sigma^j(h)\sigma^j(h) \rangle
\\ &= D\phi(y) \alpha_{\phi, \zeta}(y) + \frac{1}{2} \sum_{j \in \mathbb{N}} \big( D^2 \phi(y)(\sigma_{\phi, \zeta}^j(y), \sigma_{\phi, \zeta}^j(y)) - D\sigma^j(h) \sigma^j(h) \big).
\end{align*}
This gives us
\begin{align}\label{tangent-2}
Ah + \alpha(h) = D\phi(y) \alpha_{\phi, \zeta}(y) + \frac{1}{2} \sum_{j \in \mathbb{N}} D^2 \phi(y)(\sigma_{\phi, \zeta}^j(y), \sigma_{\phi, \zeta}^j(y)).
\end{align}
Now, let $Y$ be the strong solution to (\ref{SDE-Y}) with initial condition $y_0 := \langle \zeta, h_0 \rangle \in V$. Since $V$ is open, there exists $\epsilon > 0$ such that $\overline{B_{\epsilon}(y_0)} \subset V$. We define the stopping time
\begin{align*}
\tau := \inf \{ t \geq 0 : Y_t \notin B_{\epsilon}(y_0) \}.
\end{align*}
Since the process $Y$ has continuous sample paths and satisfies $Y_0 = y_0$, we have $\tau > 0$ and $\mathbb{P}$--almost surely
\begin{align*}
Y_{t \wedge \tau} \in V \quad \text{for all $t \geq 0$.}
\end{align*}
Therefore, defining the $H$-valued process $X := \phi(Y)$ we have $\mathbb{P}$--almost surely
\begin{align*}
X_{t \wedge \tau} \in U \cap \mathcal{M} \quad \text{for all $t \geq 0$.}
\end{align*}
Moreover, using It\^{o}'s formula (Theorem~\ref{thm-Ito}) and incorporating (\ref{tangent-1}), (\ref{tangent-2}), we obtain $\mathbb{P}$--almost surely
\begin{align*}
X_{t \wedge \tau} &= \phi(y_0) 
\\ &\quad + \int_0^{t \wedge \tau} \bigg( D\phi(Y_s)\alpha_{\phi, \zeta}(Y_s) + \frac{1}{2} \sum_{j \in \mathbb{N}} D^2 \phi(Y_s) \phi(Y_s)(\sigma_{\phi, \zeta}^j(Y_s), \sigma_{\phi, \zeta}^j(Y_s)) \bigg) ds
\\ &\quad + \sum_{j \in \mathbb{N}} \int_0^{t \wedge \tau} D \phi(Y_s) \sigma_{\phi, \zeta}^j(Y_s) d\beta_s^j
\\ &= \phi(y_0) + \int_0^{t \wedge \tau} \big( A \phi(Y_s) + \alpha(\phi(Y_s)) \big) ds + \sum_{j \in \mathbb{N}} \int_0^{t \wedge \tau} \sigma^j(\phi(Y_s)) d\beta_s^j
\\ &= h_0 + \int_0^{t \wedge \tau} \big( A X_s + \alpha(X_s) \big) ds + \sum_{j \in \mathbb{N}} \int_0^{t \wedge \tau} \sigma^j(X_s) d\beta_s^j, \quad t \geq 0,
\end{align*}
showing that $X$ is a local strong solution to (\ref{SPDE}) with lifetime $\tau$.

\noindent (3) $\Rightarrow$ (1): This implication is a direct consequence of Proposition~\ref{prop-stark-schwach}.
\end{proof}

The results from this section are closely related to the existence of finite dimensional realizations, that is, the existence of invariant manifolds for each starting point $h_0$, and we point out the articles \cite{Bj_Sv, Bj_La}, \cite{Filipovic, Filipovic-Teichmann-royal} and \cite{Tappe-Wiener, Tappe-Levy} regarding this topic. Furthermore, we mention that Theorem~\ref{thm-manifolds} has been extended in \cite{Manifolds} to SPDEs with jumps.

\end{document}